\documentclass[a4paper,UKenglish,cleveref, thm-restate]{lipics-v2021}

\pdfoutput=1 
\hideLIPIcs
\nolinenumbers

\title{Symmetry classes of Hamiltonian cycles\thanks{An extended abstract~\cite{MFCS} of this article appeared in the Proceedings of the  International Symposium on Mathematical Foundations of Computer Science (MFCS) 2025.}}
\titlerunning{Symmetry classes of Hamiltonian cycles}

\author{Júlia Baligács}{University of Oxford, United Kingdom}{jbaligacs@gmail.com}{https://orcid.org/0000-0003-2654-149X}{}

\author{Sofia Brenner}{Universität Leipzig, Germany}{sofia.brenner@uni-leipzig.de}{https://orcid.org/0009-0006-8512-2569}{Supported by a postdoctoral fellowship of the German Academic Exchange Service (DAAD) and DFG grant 522790373.}

\author{Annette Lutz}{Technische Universität Darmstadt, Germany}{lutz@mathematik.tu-darmstadt.de}{https://orcid.org/0009-0008-7699-7018}{Supported by DFG SFB/TRR 154, subproject A09.}

\author{Lena Volk}{Technische Universität Darmstadt, Germany}{volk@mathematik.tu-darmstadt.de}{https://orcid.org/0009-0004-3113-9205}{}

\authorrunning{J.~Baligács, S.~Brenner, A.~Lutz, and L.~Volk}

\Copyright{Júlia Baligács, Sofia Brenner, Annette Lutz, and Lena Volk}

\acknowledgements{We would like to thank Pascal Schweitzer for his helpful comments and support.}



\newcommand{\N}{\mathbb{N}}
\newcommand{\Z}{\mathbb{Z}}

\newcommand{\mH}{\mathcal{H}}
\DeclareMathOperator{\Aut}{Aut}
\DeclareMathOperator{\Cay}{Cay}
\DeclareMathOperator{\bx}{\square}
\DeclareMathOperator{\dcup}{\dot{\cup}}

\DeclareMathOperator{\trunc}{T}
\renewcommand{\vec}[1]{\ensuremath{\boldsymbol{#1}}} 

\usepackage{paralist}
\usepackage{mathtools}
\usepackage{adjustbox}
\usepackage[dvipsnames]{xcolor}
\usepackage{enumitem}
\setlist[itemize]{topsep=3pt, itemsep=2pt}

\usepackage{tikz}
\usetikzlibrary{bending, decorations.pathreplacing, 3d, calc, graphs, shapes.geometric}
\tikzstyle{vertex}=[circle, draw, fill=white, inner sep=0pt, minimum size=2mm]

\begin{document}

\maketitle


\begin{abstract}
	We initiate the study of Hamiltonian cycles up to symmetries of the underlying graph.
Our focus lies on the extremal case of \emph{Hamiltonian-transitive graphs}, i.e., Hamiltonian graphs where, for every pair of Hamiltonian cycles, there is a graph automorphism mapping one cycle to the other. This generalizes the extensively studied uniquely Hamiltonian graphs.
In this paper, we show that Cayley graphs of abelian groups are not Hamiltonian-transitive (under some mild conditions and some non-surprising exceptions), i.e., they contain at least two structurally different Hamiltonian cycles.
To show this, we reduce Hamiltonian-transitivity to properties of the prime factors of a Cartesian product decomposition, which we believe is interesting in its own right.
We complement our results by constructing infinite families of regular Hamiltonian-transitive graphs and take a look at the opposite extremal case by constructing a family with many different Hamiltonian cycles up to symmetry.
\end{abstract}


\section{Introduction}

Our work is motivated by a line of research studying the existence of Hamiltonian cycles in graphs that fulfill some symmetry condition.
One of the most intriguing open problems in this area is the \emph{Lovász Conjecture} \cite{Lovasz1970}: It asserts that every vertex-transitive graph, apart from five small counterexamples, contains a Hamiltonian cycle. 
Despite considerable efforts over the last decades~\cite{chen1998hamiltonicity,du2021resolving,KUT25,LovaszSurvey09,maruvsivc1987hamiltonian,muetze2016middlelevels,merino2023kneser}, this conjecture remains wide open.
Another prominent open problem in this area is \emph{Sheehan's Conjecture}~\cite{SHE75}. 
Together with results in~\cite{Thomason78,Tutte46}, its assertion is equivalent to the following:
Cycles are the only regular graphs that are uniquely Hamiltonian, i.e., contain precisely one Hamiltonian cycle.
It was proven in special cases \cite{ESF15, GOE20, UHgraphs-with-symmetries, SEA15}, but remains open in general.

One of the most important special cases for these conjectures is Cayley graphs~\cite{alspach1989lifting,du2021resolving,glover2012hamilton,pak2009survey,witte1984survey}.
While the Lovász conjecture for general Cayley graphs remains open, it is well known to hold for abelian groups (e.g., follows from~\cite{CHE81}).
In fact, Cayley graphs of abelian groups even ``exceed it'' in the sense that they are not uniquely Hamiltonian (except for cycles)~\cite{UHgraphs-with-symmetries}, i.e., they satisfy Sheehan's conjecture.
This raises the question of whether the multiple Hamiltonian cycles only exist due to many symmetries of Cayley graphs.
This is the central question of this paper and we solve it under some mild conditions.

We study Hamiltonian cycles in finite simple graphs up to symmetry.
More precisely, we consider two Hamiltonian cycles of a graph~$G$ to be equivalent, if there is an automorphism of~$G$ that maps one cycle to the other. Here we identify a cycle with its set of edges.
The focus of this paper lies on the extremal case of graphs that only contain a unique Hamiltonian cycle up to symmetry. We introduce the following notion.
\begin{definition}
	A graph $G$ is \emph{Hamiltonian-transitive} if it is Hamiltonian and, for every pair of Hamiltonian cycles, there is an automorphism of $G$ that maps one cycle to the other.
	We denote the class of all Hamiltonian-transitive graphs by~$\mH$.
\end{definition}

Note that every uniquely Hamiltonian graph is also Hamiltonian-transitive.
The intuitive difference is that we only distinguish Hamiltonian cycles that are \emph{structurally} different, which fits into a broader framework of studying combinatorial objects up to isomorphism/symmetry~\cite{KIR24, MCK98}.
For example, the complete graph is Hamiltonian transitive, but not uniquely Hamiltonian.
This raises the question of to what extent Sheehan's conjecture generalizes to Hamiltonian transitivity.

We prove that a large family of Cayley graphs of abelian groups are not Hamil\-tonian-transitive, i.e., that Sheehan's conjecture for this family generalizes.
The starting point for this is studying product decompositions of graphs. More precisely, we study how Hamiltonian transitivity behaves with respect to Cartesian products, which we believe is interesting in its own right because it reduces Hamiltonian transitivity to properties of the prime factors of a graph.

Opposed to this, we give a simple construction for an infinite family of $d$-regular Hamiltonian-transitive graphs for any $d\geq 2$.
Through the lens of Sheehan's conjecture, this highlights one of the key differences between uniquely Hamiltonian and Hamiltonian-transitive graphs.

\paragraph*{Our results}

We characterize the Hamiltonian-transitive graphs within a large family of Cayley graphs of abelian groups and Cartesian products.

Our main result is a full characterization of Hamiltonian-transitive Cayley graphs of abelian groups of odd order.
In this paper, all graphs are finite and simple, which means in the context of Cayley graphs that we only consider finite groups with a generating set that is closed under taking inverses and does not contain the identity element.
In the following, we denote by $K_n$ and \(C_n\) the complete graph and the cycle on~$n$ vertices. For a graph~$G=(V,E)$, we write $|G|:=|V|$ for its order.

\begin{restatable}{theorem}{thmoddcayleygraphs}\label{thm:odd-cayley-graphs}
	Let $G$ be the Cayley graph of an abelian group with $n:=|G|\geq 3$ odd.
	Then~$G \in \mH$ if and only if~$G \in \{K_n, C_n\}$.
\end{restatable}

For Cayley graphs of abelian groups of even order, we require somewhat different techniques and rely on a mild non-redundancy assumption on the generating set.
We obtain the following characterization, where the notation $G \bx H$ denotes the \emph{Cartesian product} and~$K_{n,n}$ denotes the complete bipartite graph in which both parts of the bipartition have size $n$.

\begin{restatable}{theorem}{thmcayleyeven}\label{thm:cayley_even}
	Let $G$ be the Cayley graph of an abelian group w.r.t.~some generating set~$S$ and assume that $n:=|G|\geq 4$ is even.
	Moreover, assume that there is some~$s \in S$ such that $S\setminus \{s,-s\}$ is non-generating.
	Then $G \in \mH$ if and only if~\hbox{$G \in \{C_n, C_4 \bx K_2, K_{3,3},  K_{4,4}\}$} or~$G=C_k \bx K_2$ where $k=n/2$ is odd.
\end{restatable}

A crucial ingredient for our proofs of Theorems~\ref{thm:odd-cayley-graphs}~and~\ref{thm:cayley_even} is the study of Cartesian products, which are interesting in their own right. First, observe that the Cartesian product of two Hamiltonian graphs is Hamiltonian again.
We show that most non-prime graphs (graphs decomposable into Cartesian products, formally defined in \cref{sec:cart-prod-general})  have, due to their symmetric structure, at least two structurally different Hamiltonian cycles.

\begin{restatable}{theorem}{thmcartprod}
	\label{thm:cart-prod}
	Let \(G\) and \(H\) be Hamiltonian graphs.
	\begin{itemize}
		\item If \(G\) and \(H\) are relatively prime, then \(G \bx H \notin \mH\).
		\item If \(G\) is prime, then \(G^{\bx n} \in \mH\) if and only if $n=1$ and $G\in \mH$.
	\end{itemize}
\end{restatable}

To extend our study of Cartesian products, we also investigate the special case of products with $K_2$ and give a full characterization of the Hamiltonian-transitive ones (\cref{thm:cartesian-product-K2}).

While these results give the impression that there are not many highly symmetric Hamiltonian-transitive graphs, we complement our findings by giving constructions of regular Hamiltonian-transitive graphs in \cref{sec:constructions}.
More precisely, we give a simple construction of an infinite class of Hamiltonian-transitive $d$-regular graphs for every $d \geq 2$  (\cref{rem:reg-graphs-in-H}).
Further, we show that truncations of cubic graphs preserve Hamiltonian transitivity, providing an interesting class of cubic 3-connected Hamiltonian-transitive graphs (\cref{thm:cub-graphs-in-H}).

These results focus on graphs with only one Hamiltonian cycle up to symmetry, and in the final part of this paper, we touch on the opposite extremal case by studying graphs with many Hamiltonian cycles up to symmetry. We provide a construction of graphs on $n$ vertices with $2^{\Theta(n \log (n))}$ symmetry classes of Hamiltonian cycles (\cref{prop:many-ham-cyc}), which matches a trivial upper bound on the number of structurally different Hamiltonian cycles.

\paragraph*{Related work}

There is a vast amount of work studying uniquely Hamiltonian graphs and Sheehan's Conjecture.
The arguably most important result for our work is that the conjecture was proven for graphs with a large automorphism group in~\cite{UHgraphs-with-symmetries}, including vertex-transitive graphs, and in particular, Cayley graphs.
Further research includes the degree sequences that uniquely Hamiltonian graphs can have (see e.g.,~\cite{UHgraphs-small-degree-vertices, UHgraphs-many-degree-sets, FLE14}) or verification of the conjecture for special graph classes, such as claw-free graphs of order not divisible by~6~\cite{ESF15}, claw-free graphs of minimum degree at least 3~\cite{SEA15}, and graphs of order up to~21~\cite{GOE20}. Bounds on the number of distinct Hamiltonian cycles in regular Hamiltonian graphs were given in~\cite{JOO24}. 

Uniquely Hamiltonian graphs as well as graphs with few Hamiltonian cycles have also been investigated from a computational point of view. Algorithms to compute a second Hamiltonian cycle in a graph were developed in~\cite{DEL24}. Further, an algorithm for the exhaustive generation of graphs with precisely $k$ Hamiltonian cycles is given in~\cite{GOE20}.  In~\cite{ITZ22}, the authors consider symmetry breaking for uniquely Hamiltonian graphs.
In~\cite{GOE24}, the authors study the number of Hamiltonian cycles in $k$-regular and $(k,l)$-regular graphs and disprove a lower bound on the number of Hamiltonian cycles in $k$-regular graphs conjectured by Haythorpe~\cite{DBLP:journals/em/Haythorpe18}. 

Motivated by the application for Gray codes,
there has recently been interest in finding Hamiltonian cycles that are particularly symmetric in the sense that they can be rotated by a graph automorphism. 
To this end, the \emph{Hamilton compression}, a parameter measuring the rotation symmetry of a Hamiltonian cycle, was introduced~\cite{GRE24} and investigated across various families of vertex-transitive graphs~\cite{KUT24,KUT25}. While the Hamilton compression measures the symmetries within a Hamiltonian cycle, we focus on the symmetries between different Hamiltonian cycles. However, the Hamilton compression is the same for all Hamiltonian cycles if the graph is Hamiltonian-transitive.

\section{Cartesian products}
\label{sec:cartprod}

Cartesian products provide a natural product decomposition of graphs, so that we begin with studying how Hamiltonian transitivity behaves with respect to such a decomposition.

More formally, the \textit{Cartesian product} of two graphs \(G\) and \(H\), denoted by \(G \bx H\), is defined as the graph with vertex set \(V(G) \times V(H)\) and two vertices \((g,h)\) and \((g',h')\) being adjacent if and only if \(g = g'\) and \(\{h, h'\} \in E(H)\), or \(h = h'\) and \(\{g, g'\} \in E(G)\).


\subsection{Cartesian products of Hamiltonian graphs}\label{sec:cart-prod-general} 

In this subsection, we focus on Hamiltonian-transitive Cartesian products of Hamiltonian graphs.
Here we rely on the assumption that the factors are Hamiltonian because this ensures that their product is Hamiltonian.
In particular, we assume that each factor has order at least three.  
Finding conditions under which~\(G \bx H\) is Hamiltonian for general graphs~\(G\) and~\(H\) is an independent research direction and remains wide open.

Next, let us introduce some of the notation in the context of Cartesian products. 
A non-trivial graph is called \emph{prime} if it cannot be decomposed into a Cartesian product of two~non-trivial graphs.
Every connected graph~\(G\) has a unique decomposition
\begin{equation}
	\label{eq:prime-factorization}
	G=H_1^{\bx r_1}  \bx \dots \bx H_k^{\bx r_k}
\end{equation}
for pairwise distinct prime graphs~\(H_1, \dots, H_k\) and \(r_1, \dots, r_k \in \N\) (see \cite[Chapter~6]{handbookprodgraphs}), where $H_i^{\bx r_i}$ denotes the Cartesian product of $r_i$ copies of $H_i$. In this case, we call the graphs~\(H_1, \dots, H_k\) \textit{prime factors} of~\(G\). Two graphs \(G\) and \(H\) are \textit{relatively prime}, if they do not share a prime factor.

With this, we can formulate the main result of this subsection.

\thmcartprod*

We split the proof of~\cref{thm:cart-prod} into two lemmata, beginning with the case that~\(G\) and~\(H\) are relatively prime.

\begin{restatable}{lemma}{lemmacartesianrelativeprime}\label{lem:prime-not-in-H}
	Let~\(G\) and~\(H\) be Hamiltonian graphs. If~\(G\) and~\(H\) are relatively prime, then~\(G \bx H \not \in \mH\).
\end{restatable}

\begin{proof}
	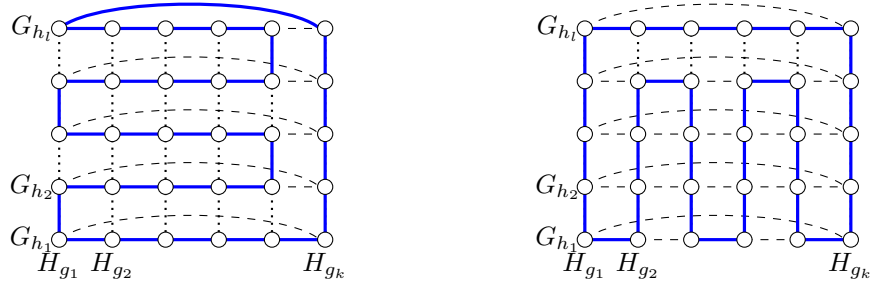
\begin{figure}
		\centering
		\begin{tikzpicture}[scale=0.7]
			\foreach \x in {0,1,2,3,4,5} {
				\foreach \y in {0,1,2,3} {
					\draw[gray] (\x,\y) to (\x,\y+1);
				}
			}
			
			\foreach \x in {0,1,2,3,4} {
				\foreach \y in {0,1,2,3,4} {
					\draw[gray] (\x,\y) to (\x+1,\y);
				}
			}
			
			\foreach \y in {0,1,2,3,4} {
				\draw[dashed, bend left=40, looseness=0.5, gray] (0,\y) to (4.9,\y);
			}
			
			\node at (-0.5,0) {\(G_{h_1}\)};
			\node at (-0.5,1) {\(G_{h_2}\)};
			\node at (-0.5,4) {\(G_{h_l}\)};
			
			\node at (0,-0.5) {\(H_{g_1}\)};
			\node at (1,-0.5) {\(H_{g_2}\)};
			\node at (5,-0.5) {\(H_{g_k}\)};
			
			\foreach \x in {0,1,2,3} {
				\foreach \y in {0,1,2,3,4} {
					\draw[blue, very thick] (\x,\y) to (\x+1,\y);
				}
			}
			\draw[blue, very thick, bend left=40, looseness=0.5] (0,4) to (4.9,4);
			\draw[blue, very thick] (0,0) to (0,1);
			\draw[blue, very thick] (4,1) to (4,2);
			\draw[blue, very thick] (0,2) to (0,3);
			\draw[blue, very thick] (4,3) to (4,4);
			
			\draw[blue, very thick] (4,0) to (5,0);
			\foreach \y in {0,1,2,3} {
				\draw[blue, very thick] (5,\y) to (5,\y+1);
			}
			
			\foreach \x in {0,1,2,3,4,5} {
				\foreach \y in {0,1,2,3,4} {
					\node [style=vertex] (\x,\y) at (\x,\y) {};
				}
			}
		\end{tikzpicture}
		\hspace{2cm}
		\begin{tikzpicture}[scale=0.7]
			\foreach \x in {0,1,2,3,4,5} {
				\foreach \y in {0,1,2,3} {
					\draw[gray] (\x,\y) to (\x,\y+1);
				}
			}
			
			\foreach \x in {0,1,2,3,4} {
				\foreach \y in {0,1,2,3,4} {
					\draw[gray] (\x,\y) to (\x+1,\y);
				}
			}
			
			\foreach \y in {0,1,2,3,4} {
				\draw[gray, dashed, bend left=40, looseness=0.5] (0,\y) to (4.9,\y);
			}
			
			\node at (-0.5,0) {\(G_{h_1}\)};
			\node at (-0.5,1) {\(G_{h_2}\)};
			\node at (-0.5,4) {\(G_{h_l}\)};
			
			\node at (0,-0.5) {\(H_{g_1}\)};
			\node at (1,-0.5) {\(H_{g_2}\)};
			\node at (5,-0.5) {\(H_{g_k}\)};
			
			\foreach \x in {0,1,2,3,4,5} {
				\foreach \y in {0,1,2} {
					\draw[blue, very thick] (\x,\y) to (\x,\y+1);
				}
			}
			\draw[blue, very thick] (0,0) to (1,0);
			\draw[blue, very thick] (2,0) to (3,0);
			\draw[blue, very thick] (4,0) to (5,0);
			\draw[blue, very thick] (0,3) to (0,4);
			\draw[blue, very thick] (5,3) to (5,4);
			\draw[blue, very thick] (1,3) to (2,3);
			\draw[blue, very thick] (3,3) to (4,3);
			\foreach \x in {0,1,2,3,4} {
				\draw[blue, very thick] (\x,4) to (\x+1,4);
			}
			
			\foreach \x in {0,1,2,3,4,5} {
				\foreach \y in {0,1,2,3,4} {
					\node [style=vertex] (\x,\y) at (\x,\y) {};
				}
			}
		\end{tikzpicture}
		\caption{Illustration of the construction in \cref{lem:prime-not-in-H}. The Hamiltonian cycle of~\(G \bx H\) on the left has~\(2(l-1)\) edges in~\(E_H\), while the Hamiltonian cycle on the right has~\(2(l-1)+(k-2)(l-2)\) edges in~\(E_H\).}
		\label{fig:can-ham-cyc-cart-prod}
	\end{figure} 
	Let~\(P := G \bx H\). We first introduce two types of edges in~\(P\). Then, we show that these edge sets are preserved under automorphisms of~\(P\). Using that there are two Hamiltonian cycles of~\(P\) containing a different number of edges from each of the respective two edge sets, we can then deduce that~\(P \not \in \mH\).
	
	For~\(v \in V(H)\) let~\(G_v \coloneq P[\{(u,v) \colon u \in V(G)\}]\) denote the subgraph of~\(P\) induced by all vertices~\((u,v)\) with~\(u \in V(G)\). Clearly, we have~\(G_v \cong G\). In the following, we call these graphs \textit{copies} of~\(G\) in~\(P\). Let~\(E_G\) denote edges of~\(P\) that are contained in a copy of~\(G\), i.e.,~\(E_G = \cup_{v \in H} E(G_v)\). Analogously define copies of~\(H\) in~\(P\) and the edge set~\(E_H\). We have~\(E(P)=E_G \dcup E_H\). (In \cref{fig:can-ham-cyc-cart-prod}, $E_G$ is the set of all horizontal edges and $E_H$ is the set of all vertical edges.)
	
	Every automorphism of~\(P\) preserves the sets~\(E_G\) and~\(E_H\) as sets: Since~\(G\) and~\(H\) are relatively prime, every automorphism \(\varphi \in \Aut(G \bx H)\) is of the form \(\varphi(u,v)=(\varphi_G(u),\varphi_H(v))\) for two automorphisms~\(\varphi_G \in \Aut(G)\) and~\(\varphi_H \in \Aut(H)\) and all \((u,v) \in G \bx H\) (see~\cite[Corollary~6.12.]{handbookprodgraphs}). Let~\(e \in E_G\). Then,~\(e=\{(u_1,v),(u_2,v)\}\) for some \(u_1,u_2 \in G\) and~\(v \in H\). Thus, we have \[\varphi(e)=\{(\varphi_G(u_1),\varphi_H(v)),(\varphi_G(u_2),\varphi_H(v))\} \in E(G_{\varphi_H(v)}) \subseteq E_G.\] Analogously, all edges from~\(E_H\) are mapped to edges of~\(E_H\).
	Let $k:=|G|$, $l:=|H|$ and fix a Hamiltonian cycle of $G$ and of $H$.
	Using these, we can construct the Hamiltonian cycle of~\(P\) depicted on the left of Figure~\ref{fig:can-ham-cyc-cart-prod}. (Note that the construction slightly changes depending on the parity of \(|H|\), but the following arguments hold for both.) This cycle has~\(2(l-1)\) edges in~\(E_H\) and~\(2(k-1)+(l-2)(k-2)\) edges in~\(E_G\).
	
	By switching the roles of~\(G\) and~\(H\) in the cycle construction, we get a second Hamiltonian cycle of~\(P\) depicted on the right of Figure~\ref{fig:can-ham-cyc-cart-prod}. This has~\(2(l-1)+(k-2)(l-2)\) edges in \(E_H\). Assume that \(P \in \mH\). Then, there is an automorphism of the graph mapping one Hamiltonian cycles to the another. Since every automorphism of~\(P\) preserves the sets~\(E_H\) and~\(E_G\), by comparing the number of edges in~\(E_H\) we get
	\(2(l-1)+(k-2)(l-2) = 2(l-1).\) For~\(l,k \in \N_{>1}\), solutions only exist for~\(k=2\) or~\(l=2\). Since we assumed~\(G\) and~\(H\) to be Hamiltonian and there are no Hamiltonian graphs on two vertices, this is a contradiction.
\end{proof}

After dealing with relatively prime graphs, the next step towards a classification is to consider the Cartesian product of a prime graph with itself. 

\begin{restatable}{lemma}{lemmacartesianpowers}\label{lem:cart-prod-power}
	Let~\(G\) be a Hamiltonian and prime graph. Then~\(G^{\bx n} \not \in \mH\) for all~\(n \geq 2\).
\end{restatable}

\begin{proof}
	Note that vertices of $G^{\bx n}$ are $n$-tuples of vertices of $G$ and two vertices $u$ and $v$ are adjacent if and only if they only differ in one component $i$ and $\{u_i,v_i\}$ is an edge of~$G$.  
	Let~$E_i$ be the set of edges~\(\{u,v\}\) of \(G^{\bx n}\) where \(u\) and \(v\) differ in their \(i\)-th component.
	Then~${E_1, \dots, E_n}$ is a partition of the edges of~\(G^{\bx n}\). 
	Since \(G\) is prime, every automorphism of $G^{\bx n}$ permutes the sets $E_1, \dots, E_n$ (\cite[Theorem 6.10]{handbookprodgraphs}), i.e., if $e, e' \in E_i$ and the automorphism maps $e$ to~$E_j$, then $e'$ is also mapped to $E_j$.
	
	Let \(k:=|G|\). First, consider the case that~\(k \geq 4\). Our strategy is as follows: We inductively~construct Hamiltonian cycles $C_n$ and \(\hat{C}_n\) in $G^{\bx n}$ (for $n\geq 2$) that cannot be mapped to each other. 
	For this, we also make use of the auxiliary statement that the constructed cycles fulfill $|E(C_n) \cap E_i|\geq 4$ and $|E(\hat{C}_n) \cap E_i|\geq 4$ for all $i \in \{1, \dots, n\}$, where by slight abuse of notation, $E_i$ is also w.r.t.~$n$.
	\begin{figure}
		\centering
		\begin{tikzpicture}[scale=0.65]
			
			\foreach \x in {0,1,2,3,4,5} {
				\foreach \y in {0,1,2,3} {
					\draw[gray] (\x,\y) to (\x,\y+1);
				}
			}
			
			\foreach \x in {0,1,2,3,4} {
				\foreach \y in {0,1,2,3,4} {
					\draw[gray] (\x,\y) to (\x+1,\y);
				}
			}
			
			\foreach \y in {0,1,2,3,4} {
				\draw[gray, dashed, bend left=40, looseness=0.5] (0,\y) to (4.9,\y);
			}
			
			\node at (-0.6,0) {\(G_{g_1}\)};
			\node at (-0.6,1) {\(G_{g_2}\)};
			\node at (-0.6,4) {\(G_{g_k}\)};
			
			\foreach \x in {0,1,2,3} {
				\foreach \y in {0,1,2,3,4} {
					\draw[blue, very thick] (\x,\y) to (\x+1,\y);
				}
			}
			
			\draw[blue, very thick, bend left=40, looseness=0.5] (0,4) to (4.9,4);
			\draw[blue, very thick] (0,0) to (0,1);
			\draw[blue, very thick] (4,1) to (4,2);
			\draw[blue, very thick] (0,2) to (0,3);
			\draw[blue, very thick] (4,3) to (4,4);
			
			\draw[blue, very thick] (4,0) to (5,0);
			\foreach \y in {0,1,2,3} {
				\draw[blue, very thick] (5,\y) to (5,\y+1);
			}
			
			\foreach \x in {0,1,2,3,4,5} {
				\foreach \y in {0,1,2,3,4} {
					\node [style=vertex] (\x,\y) at (\x,\y) {};
				}
			}
		\end{tikzpicture}
		\hspace{1cm}
		\begin{tikzpicture}[scale=0.65]
			
			\foreach \x in {0,1,2,3,4,5} {
				\foreach \y in {0,1,2,3} {
					\draw[gray] (\x,\y) to (\x,\y+1);
				}
			}
			
			\foreach \x in {0,1,2,3,4} {
				\foreach \y in {0,1,2,3,4} {
					\draw[gray] (\x,\y) to (\x+1,\y);
				}
			}
			
			\foreach \y in {0,1,2,3,4} {
				\draw[gray, dashed, bend left=40, looseness=0.5] (0,\y) to (5,\y);
			}
			\node at (-0.6,0) {\(G_{g_1}\)};
			\node at (-0.6,1) {\(G_{g_2}\)};
			\node at (-0.6,4) {\(G_{g_k}\)};
			
			\foreach \x in {0,1,3,4} {
				\foreach \y in {0,1,2,3,4} {
					\draw[blue, very thick] (\x,\y) to (\x+1,\y);
				}
			}
			
			\draw[blue, very thick, bend left=40, looseness=0.5] (0,4) to (5,4);
			\draw[blue, very thick] (0,0) to (0,1);
			\draw[blue, very thick] (5,0) to (5,1);
			\draw[blue, very thick] (0,2) to (0,3);
			\draw[blue, very thick] (5,2) to (5,3);
			
			\draw[blue, very thick] (2,0) to (3,0);
			\foreach \y in {1,3} {
				\draw[blue, very thick] (2,\y) to (2,\y+1);
				\draw[blue, very thick] (3,\y) to (3,\y+1);
			}

			\foreach \x in {0,1,2,3,4,5} {
				\foreach \y in {0,1,2,3,4} {
					\node [style=vertex] (\x,\y) at (\x,\y) {};
				}
			}
		\end{tikzpicture}		
		\caption{Illustration of the constructions for the base case in~\cref{lem:cart-prod-power}. The left picture shows the Hamiltonian cycle $C_2$, the right picture the Hamiltonian cycle $\hat{C}_2$ of \(G^{\bx 2}\).}
		\label{fig:ham-cyc-GxH}
	\end{figure}
	
	\begin{figure}
		\centering
		\begin{tikzpicture}[scale=0.65]
			
			\foreach \y in {0,1,2,3,4} {
				\draw[gray, dashed, bend left=40, looseness=0.5] (0,\y) to (5,\y);
			}
			\foreach \y in {0,1,2,3,4} {
				\draw[gray] (4,\y) to (5,\y);
			}
			\node at (-0.95,0) {\(C_{n\!-\!1}\)};
			\node at (-0.95,1) {\(C_{n\!-\!1}\)};
			\node at (-0.95,4) {\(C_{n\!-\!1}\)};
			
			\foreach \x in {0,1,2,3} {
				\foreach \y in {0,1,2,3,4} {
					\draw[blue, very thick] (\x,\y) to (\x+1,\y);
				}
			}			
			\draw[blue, very thick, bend left=40, looseness=0.5] (0,4) to (5,4);
			\draw[blue, very thick] (0,0) to (0,1);
			\draw[blue, very thick] (0,2) to (0,3);
			\draw[blue, very thick] (4,3) to (4,4);
			\draw[blue, very thick] (4,1) to (4,2);
			\draw[blue, very thick] (4,0) to (5,0);
			\foreach \y in {0,1,2,3} {
				\draw[blue, very thick] (5,\y) to (5,\y+1);
			}	
			
			\foreach \y in {0,1,2,3,4} {
				\node [circle,draw, inner sep=0.75pt, fill=white] (0,\y) at (0,\y) {\footnotesize \(u\)};
				\node [circle,draw, inner sep=0.6pt, fill=white] (4,\y) at (4,\y) {\footnotesize \(w\)};
				\node [circle,draw, inner sep=0.75pt, fill=white] (5,\y) at (5,\y) {\footnotesize \(v\)};
			}
			\foreach \x in {1,2,3} {
				\foreach \y in {0,1,2,3,4} {
					\node [style=vertex] (\x,\y) at (\x,\y) {};
				}
			}
		\end{tikzpicture}
		\caption{Illustration of the induction step in~\cref{lem:cart-prod-power}. The Hamiltonian cycle~\(C_n\) of \(G^{\bx n}\) constructed from the Hamiltonian cycle~\(C_{n-1}\) of~\(G^{\bx n-1}\), here depicted for \(|C_{n-1}|=5\).}
		\label{fig:ham-cyc-GxH_ind}
	\end{figure}
	We start with the base case, i.e., $n=2$.
	For simplicity, we call the edges in~\(E_1\) and~\(E_2\) \textit{vertical} and \textit{horizontal} edges, respectively. Let~\((g_1, \dots, g_k)\) be a Hamiltonian cycle of~\(G\).
	From this cycle, we can construct the two Hamiltonian cycles~\(C_2\) and~\(\hat{C}_2\) of~\(G^{\bx 2}\) depicted in~\cref{fig:ham-cyc-GxH}. (Note that the constructions slightly change depending on the parity of \(|G|\), but the following arguments hold for both.)
	For the sake of contradiction, assume there is an automorphism~\({\varphi \in \Aut(G \bx G)}\) mapping~\(C_2\) to~\(\hat{C}_2\). Since every vertical edge in~\(\hat{C}_2\) is preceded and followed by horizontal edges (here we use $k\geq 4$), but all horizontal edges in~\(C_2\) are either preceded or followed by another horizontal edge,~\(\varphi\) cannot swap the sets~\(E_1\) and~\(E_2\), so it preserves each. But this leads to a contradiction because there are~\(k-1\) consecutive vertical edges in~\(C_2\), but not in~\(\hat{C}_2\).
	Also, note that both cycles contain at least $2(k-1)\geq 4$ edges from $E_1$ and from $E_2$ so that this completes the base case. 
	
	For the induction step, let~\(n \geq 3\).
	Let~\(u,v\) and~\(w\) be three consecutive vertices of~\(C_{n-1}\). Let~\(C_n\) be the cycle obtained from~\(C_{n-1}\) using~\(u,v\), and~\(w\) as depicted in~\cref{fig:ham-cyc-GxH_ind}. 
	Note here that the vertical edges in \cref{fig:ham-cyc-GxH_ind} are precisely the edges in $E_n$.
	Analogously, construct~\(\hat{C}_n\) from~\(\hat{C}_{n-1}\) and three consecutive vertices~\(\hat{u},\hat{v}\) and~\(\hat{w}\) in~\(\hat{C}_{n-1}\). 
	Note that 
	\begin{equation}\label{eq:IntersectionEn}
		|E(C_n) \cap E_n| = 2(k-1)\geq 4.
	\end{equation}
	Further, since $C_n$ contains $k$ copies of the cycle $C_{n-1}$ with one or two edges removed and, by the induction hypothesis, $C_{n-1}$ uses at least 4 edges from $E_i$, we~have
	\begin{equation}\label{eq:IntersectionEi}
		|E(C_n) \cap E_i| \geq 2k \geq 4
	\end{equation}
	for all~\(i \in \{1, \dots, n-1\}\). 
	We obtain analogs of the equations \eqref{eq:IntersectionEn} and \eqref{eq:IntersectionEi} for the cycle~$\hat{C}_n$.
	It remains to prove that the cycles cannot be mapped to each other.
	Assume there is an automorphism~\({\varphi_n \in \Aut(G^{\bx n})}\) mapping~\(C_n\) to~\(\hat{C}_n\).  
	Using \eqref{eq:IntersectionEn} and \eqref{eq:IntersectionEi}, we obtain that~	$|E(C_n) \cap E_n|=2(k-1)<2k\leq |E(\hat{C}_n) \cap E_i|$ for all $i\neq n$, so that~\(\varphi\) has to preserve the set~\(E_n\).  
	Thus,~\(\varphi_n\) is of the form~\((\varphi_{n-1}, \alpha)\) for some~\(\varphi_{n-1} \in \Aut(G^{\bx n-1})\) and~\(\alpha \in \Aut(G)\).
	Note that both cycles \(C_n\) and \(\hat{C}_n\) contain precisely one path of edges in $E_n$ of length $k-1$, so that $\varphi$ maps one to the other, i.e., the path \(((v,g_1), \dots, (v,g_k))\) in \(C_n\) is mapped to the path~\(((\hat{v},g_1), \dots, (\hat{v},g_k))\) in~\(\hat{C}_n\). 
	Hence,~\(\alpha\) either maps~\(g_1\) to~\(g_1\) or, in case~\(\alpha\) reverses the order of the path, to~\(g_k\). In both cases, \(\varphi_{n-1}\) maps~\(C_{n-1}\) to~\(\hat{C}_{n-1}\), which contradicts the induction hypothesis.
	Therefore, \(C_n\) cannot be mapped to \(\hat{C}_n\) by an automorphism. This completes
	the proof in case $k\geq 4$.
	
	Next, assume~\(k=3\). Since~\(G\) is Hamiltonian, this implies~\(G=K_3\). In the following, we refine the argument for \(|G| \geq 4\).
	First, consider the Hamiltonian cycles~\(C_2\) and~\(\hat{C}_2\) of~\(K_3^{\bx 2}\) depicted in~\cref{fig:K3-powers}. They cannot be mapped to each other by automorphisms of~\(K_3^{\bx 2}\), since~\(C_2\) has three horizontal edges, while~\(\hat{C}_2\) has at least four edges in each direction. This shows that~\(K_3^{\bx 2} \not \in \mH\).
	
	Now, let~\(n \geq 3\). Consider the Hamiltonian cycles~\(C_n\) and~\(\hat{C}_n\) of~\(K_3^{\bx n}\) obtained from~\(C_{n-1}\) and~\(\hat{C}_{n-1}\) as before. Then,~\(C_n\) and~\(\hat{C}_n\) again have precisely four edges in~\(E_n\). Further, they have more than four edges in all~\(E_i\) for~\(i < n\):
	From the three copies of \(C_{n-1}\) (resp. \(\hat{C}_{n-1}\)) in the construction of \(C_n\) (resp. \(\hat{C}_{n}\)), we in total remove each of the edges \(\{v,w\}\) and \(\{u,v\}\) twice. Thus, for \(\hat{C}_{n-1}\), we have at least \(3 \cdot 4 - 4 = 8\) edges in each \(E_i\). For \(C_n\), note that there are no two consecutive horizontal edges in \(C_2\) and thus, \(C_3\) contains at least \(3 \cdot 3 - 2 = 7\) edges in each direction. Hence, \(C_n\) contains at least 7 edges in each \(E_i\). With this, the proof can be carried out as in the case $|G|\geq 4$.
\end{proof}

\begin{figure}
	\centering
	\begin{tikzpicture}
		\foreach \x in {0,1,2}{
			\foreach \y in {0,1,2}{
				\node[style=vertex] (\x\y) at (\x,\y) {};
			}
		}
		\foreach \x in {0,1,2}{
			\foreach \y in {0,1,2}{
				\ifnum\x<2
				\draw[gray] (\x\y) -- (\the\numexpr\x+1\relax\y);
				\fi
				\ifnum\x>1
				\draw[gray, bend right] (2\y) to (0\y);
				\fi
				\ifnum\y<2
				\draw[gray] (\x\y) -- (\x\the\numexpr\y+1\relax);
				\fi
				\ifnum\y>1
				\draw[gray, bend right] (\x0) to (\x2);
				\fi
			}
		}
		\draw[blue, very thick] (00) to (01);
		\draw[blue, very thick] (01) to (02);
		\draw[blue, very thick, bend left] (02) to (22);
		\draw[blue, very thick, bend left] (22) to (20);
		\draw[blue, very thick] (20) to (21);
		\draw[blue, very thick] (21) to (11);
		\draw[blue, very thick] (11) to (12);
		\draw[blue, very thick, bend left] (12) to (10);
		\draw[blue, very thick] (10) to (00);
	\end{tikzpicture}
	\hspace{1cm}
	\begin{tikzpicture}
		\foreach \x in {0,1,2}{
			\foreach \y in {0,1,2}{
				\node[style=vertex] (\x\y) at (\x,\y) {};
			}
		}
		
		\foreach \x in {0,1,2}{
			\foreach \y in {0,1,2}{
				\ifnum\x<2
				\draw[gray] (\x\y) -- (\the\numexpr\x+1\relax\y);
				\fi
				\ifnum\x>1
				\draw[gray, bend right] (2\y) to (0\y);
				\fi
				\ifnum\y<2
				\draw[gray] (\x\y) -- (\x\the\numexpr\y+1\relax);
				\fi
				\ifnum\y>1
				\draw[gray, bend right] (\x0) to (\x2);
				\fi
			}
		}
		
		\draw[blue, very thick] (00) to (01);
		\draw[blue, very thick] (01) to (02);
		\draw[blue, very thick, bend left] (02) to (22);
		\draw[blue, very thick] (22) to (12);
		\draw[blue, very thick] (12) to (11);
		\draw[blue, very thick] (11) to (21);
		\draw[blue, very thick] (21) to (20);
		
		\draw[blue, very thick] (20) to (10);
		\draw[blue, very thick] (10) to (00);
	\end{tikzpicture}
	\caption{Two Hamiltonian cycles of~\(K_3^{\bx 2}\). The Hamiltonian cycle~\(C_2\) on the left has only three horizontal edges, while the cycle~\(\hat{C}_2\) on the right has at least four edges in every direction.}
	\label{fig:K3-powers}
\end{figure}
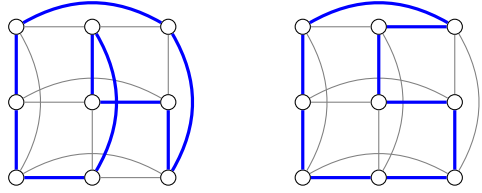 

\cref{thm:cart-prod} follows now immediately by combining \cref{lem:prime-not-in-H} and \cref{lem:cart-prod-power}.
In case a graph decomposes into a Cartesian product of Hamiltonian graphs, we can deduce the following characterization of Hamiltonian-transitive graphs.

\begin{restatable}{corollary}{corcartesiandecomposition}\label{cor:cart-prod-decomp}
	Let~\(G\) be a graph with decomposition~\(G=H_1^{\bx r_1}  \bx \dots \bx H_k^{\bx r_k}\) for pairwise distinct prime graphs \(H_1, \dots, H_k\). Assume that \(H_1,\dots,H_k \) are Hamiltonian. Then \(G \in \mH\) if and only if~\(G=H_1 \in \mH\).
\end{restatable}

\begin{proof}
	Assume \(G \in \mH\). Then,~\cref{lem:prime-not-in-H} implies that~\(G=H_1^{\bx r_1}\). Further, \cref{lem:cart-prod-power} yields that~\(r_1=1\) and hence,~\(G=H_1\).
\end{proof}


\subsection{Cartesian products with \(K_2\)}\label{sec:cartesian-K2}

In the previous subsection, we assumed that all factors of the Cartesian product are Hamiltonian to ensure that the obtained graph is Hamiltonian as well.
In this subsection, we extend our study to Cartesian products with $K_2$, which later turns out to be a crucial case for our work on Cayley graphs. More precisely, we prove the following result.

\begin{theorem}\label{thm:cartesian-product-K2}
	Let~\(G\) be a Hamiltonian graph. Then~\(G \bx K_2 \in \mH\) if and only if~\(G = C_k\) for~\(k \geq 3\) odd or~$G=C_4$.
\end{theorem}

Throughout this section, we let $V(G):=\{v_0, \dots, v_{k-1}\}$ such that $(v_0, \dots, v_{k-1})$ is a Hamiltonian cycle of~$G$. 
We denote the Cartesian product with \(K_2\) by taking an isomorphic copy of~$G$ on vertices $\{u_0, \dots, u_{k-1}\}$ and letting~$v_i$ neighboring~$u_i$ for all $i \in \{0, \dots, k-1\}$ (see \cref{fig:prism}).
We call the subgraph induced by $\{v_0, \dots, v_{k-1}\}$ the \emph{outer layer} and the subgraph induced by $\{u_0, \dots, u_{k-1}\}$ the \emph{inner layer}.

We split the proof of \cref{thm:cartesian-product-K2} into several lemmata.
First, we show that, for~$G=C_k$ with~$k$ odd, the Cartesian product~$G \bx K_2$ is indeed Hamiltonian-transitive.

\begin{restatable}{lemma}{lemmaoddcycleprism}\label{lem:oddcycle_prism}
	Let $k\geq 3$ be odd. Then $C_k \bx K_2 \in \mH$.
\end{restatable}

\begin{proof}
	Consider the cut-set $F=\bigl\{\{u_i, v_i\}: i \in \{0,\dots, k-1\}\bigl\}$. 
	We claim that every Hamiltonian cycle~$C$ in the~$k$-prism uses precisely two edges of $F$, and they are consecutive (i.e., $\{u_i, v_i\}$ and $\{u_{j}, v_{j}\}$ for $i,j \in \{0,\dots, k-1\}$ with $j=i+1 \bmod k$).
	To see this, observe that the number of edges in~$F$ that~$C$ uses is even and at least~2.
	Since~$|F|=k$ and~$k$ is odd, there is an edge in~$F$ not used by~$C$. In particular, there is an edge $\{u_i, v_i\} \in E(C)$ such that $\{u_j, v_j\} \notin E(C)$ (where $j=i+1 \bmod k$ and $i,j \in \{0, \dots, k-1\}$).
	Without loss of generality, assume that~$i=0$ and~$j=1$.
	Since every vertex has degree 3 in the prism and degree 2 in $C$, the fact that~$C$ does not contain $(u_1, v_1)$ implies that $C$ contains~$\{u_0, u_1\}, \{u_1, u_2\}, \{v_0, v_1\}, \{v_1, v_2\}$ (see \cref{fig:prism} for illustration).
	Now, if $k>3$, $C$ cannot contain $\{u_2, v_2\}$ because then $C=(u_0,u_1, u_2, v_2, v_1,v_0)$, i.e., $C$ does not contain all of the vertices. 
	Therefore, $C$ also contains $\{u_2, u_3\}, \{v_2, v_3\}$.
	Repeating this argument shows that the only edges in $F$ contained in $C$ are $\{u_0, v_0\}$ and $\{u_{k-1}, v_{k-1}\}$. This completes the proof of the claim.
	In particular, it shows that every Hamiltonian cycle in the $k$-prism can be mapped to the cycle~$(v_0, v_1, \dots, v_{k-1}, u_{k-1}, \dots, u_0)$ via rotation, which is obviously an automorphism. Therefore, the $k$-prism is in $\mH$ when $k$ is odd.
\end{proof}

\begin{figure}
	\centering
	\begin{tikzpicture}[scale=0.8]
		\foreach \i in {1,...,7} {
			\node[circle,draw, inner sep=0.5pt] (A\i) at (360/7 * \i:1.1) {\footnotesize{$u_{\the\numexpr\i-1\relax}$}};
		}
		\foreach \i in {1,...,7} {
			\node[circle, draw, inner sep=0.5pt] (B\i) at (360/7 * \i:2) {\footnotesize{$v_{\the\numexpr\i-1\relax}$}};
		}
		\foreach \i in {1,...,7} {
			\pgfmathtruncatemacro{\nexti}{mod(\i,7)+1}
			\draw[gray] (A\i) -- (A\nexti);
		}
		\foreach \i in {1,...,7} {
			\pgfmathtruncatemacro{\nexti}{mod(\i,7)+1}
			\draw[gray] (B\i) -- (B\nexti);
		}
		\foreach \i in {1,...,7} {
			\draw[gray] (A\i) -- (B\i);
		}
		\draw[very thick, blue] (A1) to (A2) to (A3) to (A4) to (A5) to (A6) to (A7) to (B7) to (B6) to (B5) to (B4) to (B3) to (B2) to (B1) to (A1);
	\end{tikzpicture}
	\caption{The up to symmetry unique Hamiltonian cycle in $C_7 \bx K_2$.}\label{fig:prism}
\end{figure}
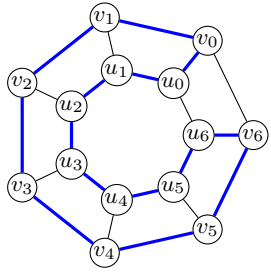

In the next two lemmata, we argue that in all other cases (where $|G|\geq 5$), the Cartesian product $G \bx K_2$ is not in $\mH$.
The difficulty is that we cannot assume any non-adjacency between vertices of $G$. More precisely, we only use that $C_k \bx K_2$ is a subgraph of $G \bx K_2$ and that $u_i$ is non-adjacent to $v_j$ if $i \neq j$.
First, we settle the case when $|G|$ is even.

\begin{restatable}{lemma}{lemmacartesianktwoeven}\label{lem:cartesian-K2-even}
	Let $G$ be Hamiltonian with $k:=|G|\geq 6$ even. Then $G \bx K_2 \notin \mH$.
\end{restatable}

\begin{proof}
	Recall that we order the vertices of $G$ such that $(v_0, \dots, v_{k-1})$ is a Hamiltonian cycle of $G$.
	We construct two Hamiltonian cycles in $G \bx K_2$ and show that there is no automorphism mapping them to each other.
	We set
	\begin{align*}
		C_1 &= (u_0, u_1, v_1, v_2, u_2, u_3, v_3, v_4, \dots, u_{k-2}, u_{k-1}, v_{k-1}, v_0)\\
		C_2 &= (u_0, \dots, u_{k-1}, v_{k-1}, \dots, v_0).
	\end{align*}
	The cycle $C_1$ is illustrated on the left side of \cref{fig:cartesian-K2-even} and $C_2$ in the middle of \cref{fig:cartesian-K2-even}.
	Note that $C_2$ is the same cycle-construction as in \cref{lem:oddcycle_prism} and $C_1$ uses that~$k$ is even.
	
	Now, observe that the cycle $C_1$ has the following property~$(*)$: 
	For every vertex~$w$, let~$w_1, w_2$ denote the two vertices that have distance 3 on the cycle to $w$. 
	Then, in $G \bx K_2$, $w$ is adjacent to exactly one of $w_1, w_2$.
	Note that this property is preserved by an automorphism.
	
	Assume for the sake of contradiction that $G \bx K_2 \in \mH$. Then $C_2$ also has property~$(*)$.
	Note that $v_0$ and $u_2$ have distance 3 on the cycle but are not adjacent. By property $(*)$, it follows that there is an edge between $v_0$ and $v_3$ (since $k>4$, we have $v_3 \neq v_{k-1}$ so that this edge is not already contained in $C_2$).
	Now, consider the cycle
	\begin{equation*}
		C_3 = (u_{k-1}, u_0, \dots, u_{k-2}, v_{k-2}, \dots, v_0, v_{k-1}),
	\end{equation*}
	which is illustrated on the right side of \cref{fig:cartesian-K2-even}.
	Note that, in $C_3$, the two vertices at distance~3 to~$v_0$ are~$u_0$ and~$v_3$, and $v_0$ is adjacent to both of them.
	Therefore, $C_3$ does not have the property $(*)$.
	This gives a contradiction so that $G \bx K_2 \notin \mH$.
\end{proof}

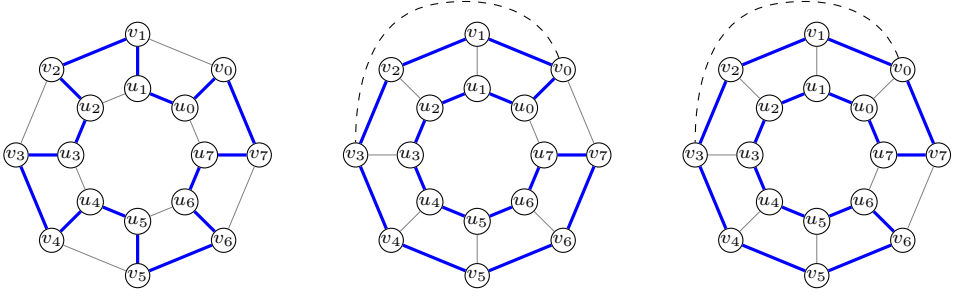
\begin{figure}
	\centering
	\begin{tikzpicture}[scale=0.8, every node/.style={inner sep=0.1pt}]
		\foreach \i in {1,...,8} {
			\node[circle,draw] (A\i) at (360/8 * \i:1.1) {\scriptsize{$u_{\the\numexpr\i-1\relax}$}};
		}
		\foreach \i in {1,...,8} {
			\node[circle, draw] (B\i) at (360/8 * \i:2) {\scriptsize{$v_{\the\numexpr\i-1\relax}$}};
		}
		\foreach \i in {1,...,8} {
			\pgfmathtruncatemacro{\nexti}{mod(\i,8)+1}
			\draw[gray] (A\i) -- (A\nexti);
		}
		\foreach \i in {1,...,8} {
			\pgfmathtruncatemacro{\nexti}{mod(\i,8)+1}
			\draw[gray] (B\i) -- (B\nexti);
		}
		\foreach \i in {1,...,8} {
			\draw[gray] (A\i) -- (B\i);
		}
		\draw[very thick, blue] (A1) to (A2) to (B2) to (B3) to (A3) to (A4) to (B4) to (B5) to (A5) to (A6) to (B6) to (B7) to (A7) to (A8) to (B8) to (B1) to (A1);
	\end{tikzpicture}
	\hspace{0.7cm}
	\begin{tikzpicture}[scale=0.8, every node/.style={inner sep=0.1pt}]
		\foreach \i in {1,...,8} {
			\node[circle,draw] (A\i) at (360/8 * \i:1.1) {\scriptsize{$u_{\the\numexpr\i-1\relax}$}};
		}
		\foreach \i in {1,...,8} {
			\node[circle, draw] (B\i) at (360/8 * \i:2) {\scriptsize{$v_{\the\numexpr\i-1\relax}$}};
		}
		\foreach \i in {1,...,8} {
			\pgfmathtruncatemacro{\nexti}{mod(\i,8)+1}
			\draw[gray] (A\i) -- (A\nexti);
		}
		\foreach \i in {1,...,8} {
			\pgfmathtruncatemacro{\nexti}{mod(\i,8)+1}
			\draw[gray] (B\i) -- (B\nexti);
		}
		\foreach \i in {1,...,8} {
			\draw[gray] (A\i) -- (B\i);
		}
		\draw[very thick, blue] (A1) to (A2) to (A3) to (A4) to (A5) to (A6) to (A7) to (A8) to (B8) to  (B7) to (B6) to (B5) to (B4) to (B3) to (B2) to (B1) to (A1);
		\draw[dashed] (B1) to [in=10, out=110] (-0.7,2.5) to [in=90, out=190] (B4);
	\end{tikzpicture}
	\hspace{0.7cm}
	\begin{tikzpicture}[scale=0.8, every node/.style={inner sep=0.1pt}]
		\foreach \i in {1,...,8} {
			\node[circle,draw] (A\i) at (360/8 * \i:1.1) {\scriptsize{$u_{\the\numexpr\i-1\relax}$}};
		}
		\foreach \i in {1,...,8} {
			\node[circle, draw] (B\i) at (360/8 * \i:2) {\scriptsize{$v_{\the\numexpr\i-1\relax}$}};
		}
		\foreach \i in {1,...,8} {
			\pgfmathtruncatemacro{\nexti}{mod(\i,8)+1}
			\draw[gray] (A\i) -- (A\nexti);
		}
		\foreach \i in {1,...,8} {
			\pgfmathtruncatemacro{\nexti}{mod(\i,8)+1}
			\draw[gray] (B\i) -- (B\nexti);
		}
		\foreach \i in {1,...,8} {
			\draw[gray] (A\i) -- (B\i);
		}
		\draw[very thick, blue] (A8) to (A1) to (A2) to (A3) to (A4) to (A5) to (A6) to (A7) to  (B7) to (B6) to (B5) to (B4) to (B3) to (B2) to (B1) to (B8) to (A8);
		\draw[dashed] (B1) to [in=10, out=110] (-0.7,2.5) to [in=90, out=190] (B4);
	\end{tikzpicture}
	\caption{Illustration of the construction in \cref{lem:cartesian-K2-even}.
		The possible edges inside a layer are omitted.
		Every vertex in the left Hamiltonian cycle is adjacent to precisely one of the vertices at distance three in the cycle. It is impossible that~$v_0$ has this property in both, the cycle in the middle and on the right. The cycle on the left only exists when~$|G|$ is even.}\label{fig:cartesian-K2-even}
\end{figure}

Next, we consider the case where $|G|$ is odd but~$G$ is not a cycle.
Note that we need a different strategy here because the cycle we used in \cref{lem:cartesian-K2-even} (left side of \cref{fig:cartesian-K2-even}) only exists when~$|G|$ is even.

\begin{restatable}{lemma}{lemmacartesianoddnoncycle}\label{lem:cartesian-K2_odd_noncycle}
	Let $G$ be Hamiltonian with $|G|=:k$ odd and $G\neq C_k$. 
	Then $G \bx K_2 \notin \mH$.
\end{restatable}

\begin{proof}
	For the sake of contradiction, assume that $G \bx K_2 \in \mH$.
	Note that $G$ is a cycle of length $k$ with at least one additional edge $e=\{u,v\}$. Since~$k$ is odd, the distance between~$u$ and $v$ along the cycle is even in one direction and odd in the other direction.
	Therefore, we can renumber the vertices of $G$ such that there is an edge between $v_0$ and $v_j$, where~$j\in \{3, \dots, k-2\}$ is odd.
	
	As a first step, we show the following (under the assumption that $G \bx K_2 \in \mH$).
	
	\begin{claim}
		For every $i \in \{0, \dots, k-1\}$, there is an edge between $v_i$ and $v_{i+j \bmod k}$.
	\end{claim}
	
	\begin{proof}[Proof of the claim]
		We consider cycles of the form
		\begin{equation*}
			C^{(i)}:=(u_i, u_{i+1 \bmod k}, \dots, u_{i+k-1 \bmod k}, v_{i+k-1 \bmod k}, v_{i+k-2 \bmod k}, \dots, v_i)
		\end{equation*}
		for $i \in \{0, \dots, k-1\}$ (as illustrated in the middle and on the right of \cref{fig:cartesian-K2-even}).
		
		Given an edge $e=\{u,v\}$, we say that the distance of a vertex $w$ to $e$ is \[\min\{d(u,w), d(v,w)\}, \]
		where $d$ denotes the usual graph distance. 
		We say that an edge $e$ is a \emph{symmetry edge of a Hamiltonian cycle $C$} if it is contained in~$C$ and, for every $l \in \{0, \dots, k-1\}$, the two vertices that have distance $l$ to $e$ in~$C$ are adjacent in $G\bx K_2$.
		Note that $\{u_i, v_i\}$ and~$\{u_{i+k-1 \bmod k}, v_{i+k-1 \bmod k}\}$ are symmetry edges of $C^{(i)}$.
		Moreover, we argue that there are no other symmetry edges of $C^{(i)}$: To see this, fix some other edge~$e$ of~$C^{(i)}$. W.l.o.g., we let $i=0$ and $e=\{v_p, v_{p+1}\}$ for some $p\in\{0, \dots, k-2\}$.
		Since~$k$ is odd, we have~$d(v_0, v_p)=p \neq k-p -2 = d(v_{k-1}, v_{p+1})$. Without loss of generality, let $p<k-p -2$, i.e., $p\leq (k-1)/2 -2$ and set $l:=p+1$.
		Then the two vertices at distance $l$ to $e$ are~$u_0$ and~$v_{2p+2}\in \{v_{p+1}, \dots, v_{k-1}\}$. Since these are non-adjacent, $e$ is not a symmetry edge of~$C^{(0)}$.
		
		Since we assume that $G\bx K_2 \in \mH$, there is an automorphism mapping $C^{(0)}$ to~$C^{(i)}$ for every $i\in \{0, \dots, k-1\}$.
		Note that this automorphism maps the symmetry edges of~$C^{(0)}$ to the symmetry edges of~$C^{(i)}$.
		Recall that there is an edge between $v_0$ and~$v_j$, and between $u_0$ and~$u_j$.
		This means that the end vertices of the symmetry edge $\{u_0, v_0\}$ of $C^{(0)}$ are adjacent to the vertices at distance $j$ in the cycle, where the cycle-distance here is measured such that the symmetry edge is not used.
		If there is also an edge between $v_{k-1}$ and $v_{k-1-j}$, the existence of the automorphisms mapping~$C^{(0)}$ to~$C^{(i)}$ for every $i \in \{1, \dots, k-1\}$ immediately implies that, for every~$i$,~$v_i$ is adjacent to $v_{i+j\bmod k}$ and $v_{i-j\bmod k}$, i.e., the claim holds.
		If there is no edge between  $v_{k-1}$ and~$v_{k-1-j}$, 
		the existence of the automorphism mapping $C^{(0)}$ to~$C^{(j+1)}$ implies~$v_{j+1}$ is not connected to $v_{2j+1 \bmod k}$ as $v_0$ is connected to $v_j$. Iteratively, the automorphism from~$C^{(0)}$ to $C^{(l(j+1))}$ yields $v_{l(j+1) \bmod k}$ and $v_{(l+1)(j+1)-1 \bmod k}$ are connected if and only if $l$ is even. 
		Thus, for $l=k-1$ we derive $v_{k-j-1}$ and $v_{k-1}$ are connected, contradicting our assumption.
		This completes the proof of the claim.
	\end{proof}
	
	Now that we have established the claim, we construct a second Hamiltonian cycle and argue that it cannot be mapped via an automorphism to $C^{(0)}$.
	We set
	\begin{equation*}
		C= (u_0, v_0, v_1, u_1, u_2 \dots, v_{j-1}, v_j, v_{j+1}, \dots v_{k-1}, u_{k-1}, \dots, u_j),
	\end{equation*}
	where we have used that $j$ is odd. This cycle is illustrated in \cref{fig:cartesian-K2_odd_noncycle}.
	
	Note that the cycle $C^{(0)}$ has the property that every vertex $v$ is adjacent to some other vertex $v'$ with $d_{C^{(0)}}(v,v')=j$.
	We argue that vertex $w=u_{(j+1)/2}$ does not have this property in $C$: 
	If $j \bmod 4=1$ (i.e., $(j+1)/2$ is odd), the two adjacent vertices to $w$ are $v_1$ and~$v_{j+2 \bmod k}$ (the latter is $v_{j+2}$ if $j+2<k$ or it is $v_0$ in case $j+2=k$, note that $j+1=k$ is not possible as $k$ and $j$ are both odd). Both of these vertices are not adjacent to $w$.
	In case $j \bmod 4=3$ (i.e., $(j+1)/2$ is even), the two vertices at distance $j$ in $C$ are $v_0$ and $v_{j+1}$, both of which are not adjacent to $w$.
	This proves that there is no automorphism mapping $C$ to~$C^{(0)}$, which gives a contradiction.
\end{proof}

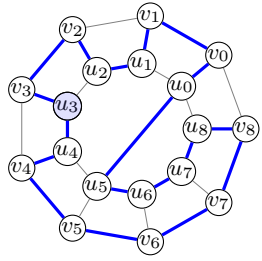
\begin{figure}
	\centering
	\begin{tikzpicture}[scale=0.8, every node/.style={inner sep=0.1pt}]
		\foreach \i in {1,...,9} {
			\node[circle,draw] (A\i) at (360/9 * \i:1.1) {\footnotesize{$u_{\the\numexpr\i-1\relax}$}};
		}
		\node[circle,draw, thick,fill=blue!60!white, fill opacity=0.3] (A4) at (360/9 * 4:1.1) {\footnotesize{$u_{3}$}};
		\foreach \i in {1,...,9} {
			\node[circle, draw] (B\i) at (360/9 * \i:1.9) {\footnotesize{$v_{\the\numexpr\i-1\relax}$}};
		}
		\foreach \i in {1,...,9} {
			\pgfmathtruncatemacro{\nexti}{mod(\i,9)+1}
			\draw[gray] (A\i) -- (A\nexti);
		}
		
		\foreach \i in {1,...,9} {
			\pgfmathtruncatemacro{\nexti}{mod(\i,9)+1}
			\draw[gray] (B\i) -- (B\nexti);
		}
		\foreach \i in {1,...,9} {
			\draw[gray] (A\i) -- (B\i);
		}
		\draw[very thick, blue] (A1) to (B1) to (B2) to (A2) to (A3) to (B3) to (B4) to (A4) to (A5) to (B5) to (B6) to (B7) to (B8) to (B9) to (A9) to (A8) to (A7) to (A6) to (A1);
	\end{tikzpicture}
	\caption{Illustration of the construction in \cref{lem:cartesian-K2_odd_noncycle}. The possible edges inside the layers are omitted. 
		The cycle uses the existence of one additional edge inside a layer $\{u_0, u_5\}$. The marked vertex $u_3$ is not adjacent to the vertices at distance 5 along the cycle ($v_1$ and $v_7$).}\label{fig:cartesian-K2_odd_noncycle}
\end{figure}

The only case that is left to settle for \cref{thm:cartesian-product-K2} is when $|G|=4$.
Note that there are only 3 Hamiltonian graphs on 4 vertices and checking these (either by hand or computation), we obtain that $G=C_4$ is the only case where $G \bx K_2 \in \mH$.

Now, we have all ingredients in place to prove \cref{thm:cartesian-product-K2}.

\begin{proof}[Proof of \cref{thm:cartesian-product-K2}]
	Let $G$ be a Hamiltonian graph (in particular, $k:=|G|\geq 3$) such that~$G \bx K_2 \in \mH$.
	By \cref{lem:cartesian-K2-even} and \cref{lem:cartesian-K2_odd_noncycle}, we obtain that either $k=4$ or $G=C_k$ for some odd $k$.
	In case $k=4$, we argued that $G=C_4$.
	This completes one implication of \cref{thm:cartesian-product-K2}.
	For the other direction, in \cref{lem:oddcycle_prism}, we have proved that~$C_k \bx K_2 \in \mH$ and have seen that $C_4 \bx K_2 \in \mH$. This completes the proof of \cref{thm:cartesian-product-K2}.
\end{proof}

In the following section, we consider Cayley graphs of abelian groups and as a base case we separately investigate the special case of cyclic groups. Therefore, let us comment on how the graphs studied in this section relate to that.

\begin{remark}
	Note that $C_k \bx K_2= \Cay(\Z_k \bx \Z_2, \{(0,1), (1,0), (0,-1), (-1,0)\})$. 
	In particular, it is the Cayley graph of an abelian group.
	If $k$ is odd, $C_k \bx K_2$ is generated by~$(1,1)$ so that  $C_k \bx K_2$ is even the Cayley graph of a cyclic group.
	On the other hand, note that for $k$ even, $C_k \bx K_2$ is not the Cayley graph of a cyclic group: To see this, assume that~$C_k \bx K_2 = \Cay(\Z_{2k}, S)$ for some generating set $S$. Since  $C_k \bx K_2 = \Cay(\Z_{2k}, S)$ is 3-regular,~$S$ contains the only self-inverse element of $\Z_{2k}$, i.e., $k$, and only one other element~$s$. 
	Then~$s$ is odd, because otherwise $S$ is not generating.
	Observe that $(0, s, 2s, \dots, ks, ks+k)$ is a closed walk of length $k+1$ because \(k\cdot s +k \equiv (s+1)k \equiv 0 \bmod 2k\). 
	This means that~$\Cay(\Z_{2k}, S)$ contains a closed walk of odd length.
	Since the $k$-prism for $k$ even is bipartite with bipartition $A=\{u_i: i \text{ odd}\} \cup \{v_i: i \text{ even}\}$, $B=V\setminus A$, this gives a contradiction.
\end{remark}


\section{Cayley graphs of abelian groups}
\label{sec:cayley}

In this section, we study Cayley graphs of abelian groups with the goal to find a characterization of the Hamiltonian-transitive ones.
Recall that it is an open question whether all Cayley graphs with at least three vertices are Hamiltonian, but this is well known to hold for Cayley graphs of abelian groups~\cite{CHE81}.
In this section, we extend this result and show that many Cayley graphs of abelian groups even contain multiple Hamiltonian cycles up to symmetry (\cref{thm:odd-cayley-graphs} and \cref{thm:cayley_even}).

\subsection{Preliminaries}\label{sec:cayley-preliminaries}

In this subsection, we collect some preliminary notation and results. 
In the following, we write all groups additively.
Given a generating set~$S$ of an abelian group $\Gamma$, we write $\Cay(\Gamma,S)$ for the Cayley graph of $\Gamma$ with respect to~$S$. 
As we restrict ourselves to undirected graphs without loops, we always assume $S$ to be closed under taking inverses and $0 \notin S$.

We begin by observing that Cayley graphs are product-preserving in the sense of the lemma below.
For this, for two groups $\Gamma_1$ and $\Gamma_2$, $\Gamma_1 \times \Gamma_2$ denotes their direct product.
By slight abuse of notation, we allow considering a set~$S_1 \subseteq \Gamma_1$ as a subset of~$\Gamma_1 \times \Gamma_2$ by actually referring to the set~$\{(s,0):s \in S_1\}$, and analogously for~$S_2$.
The following lemma is a direct consequence of the definition of Cayley graphs and Cartesian products. Although this statement is well known, we give a short proof for completeness.

\begin{restatable}{lemma}{splittingcayley}\label{lemma:splittingcayley}
	Let \(\Gamma_1\) and \(\Gamma_2\) be finite groups with generating sets \(S_1\) and \(S_2\), respectively. For \(\Gamma = \Gamma_1 \times \Gamma_2\), we have \(\Cay(\Gamma,S_1 \cup S_2) \cong \Cay(\Gamma_1, S_1) \bx \Cay(\Gamma_2, S_2)\). 
\end{restatable}

\begin{proof}
	The graphs $\Cay(\Gamma,S_1 \cup S_2)$ and $\Cay(\Gamma_1, S_1) \bx \Cay(\Gamma_2, S_2)$ have the same vertex set by definition.
	Two vertices $(u,v), (u',v')$ are adjacent in $\Cay(\Gamma, S_1 \cup S_2)$ if and only if $(u-u', v-v')$ is in $S_1$ or $S_2$. 
	This is the case if and only if $u-u'=0$ and~$v-v' \in S_2$, or $u-u'\in S_1$ and~$v-v'=0$, i.e., if and only if $(u,v), (u',v')$ are adjacent in $\Cay(\Gamma_1, S_1) \bx \Cay(\Gamma_2, S_2)$.
\end{proof}

A useful parameter in the context of Cayley graphs is the \emph{Hamilton compression}, introduced by Gregor et al.~in \cite{GRE24}. A Hamiltonian cycle $C=(v_0, \dots, v_{n-1})$ of a graph $G$ on $n$ vertices is called \emph{$k$-symmetric} if~$k$ divides $n$ and $v_i \mapsto v_{i+n/k \bmod n}$ is an automorphism of~$G$.
We call this automorphism a \emph{rotation of $C$ by $n/k$.}
Note that, if a cycle is $k$-symmetric, i.e., can be rotated by~$n/k$, it can also be rotated by multiples of~$n/k$ as this corresponds to repeatedly applying the induced automorphism.
In particular, a $k$-symmetric cycle can be rotated in both directions by $n/k$ because a rotation by~$n/k$ in one direction corresponds to a rotation by~$(k-1)n/k$ in the other direction.
The \emph{Hamilton compression of~$C$} is then defined by~\({\kappa(C):=\max\{k: \text{$C$ is $k$-symmetric}\}}.\)
For a Hamiltonian graph $G$, the \emph{Hamilton compression of $G$} is defined by \[\kappa(G):=\max\{k: \text{there is a $k$-symmetric Hamiltonian cycle of $G$}\}.\]
Note that, by definition,~$\kappa(G)$ divides~$n$.

Since the Hamilton compression of a Hamiltonian cycle is preserved under applying automorphisms, we immediately obtain the following result.

\begin{lemma}\label{lemma:samekappa}
	Let $G$ be a graph. If $G \in \mH$, then all Hamiltonian cycles of $G$ have the same Hamilton compression.
\end{lemma}

We later use the next result to split Cayley graphs into Cartesian products. 

\begin{restatable}{lemma}{compositeorder}\label{lemma:compositeorder}
	Let $\Gamma$ be an abelian group and \(p\) be a prime divisor of \(|\Gamma|\). Let \(S\) be a generating set of $\Gamma$ such that the order of each element in $S$ is either a power of $p$ or coprime to $p$.	
	Setting \(S_p:= \{s \in S \colon p \mid \operatorname{ord}(s)\}\) 
	and \(S'_{p} := S \setminus S_p\), we have~\(\Gamma = \langle S_p \rangle \times \langle S'_{p} \rangle\).
\end{restatable}

\begin{proof}
	Note that, for every $s \in S_p$, the order of $s$ is a power of $p$.
	Since the order of a sum of elements in an abelian group divides the least common multiple of the orders of the summands, we obtain that 
	the order of each element in $\langle S_p \rangle$ is also a power of~$p$.
	Similarly,~$S'_{p}$ only contains elements of order coprime to $p$ so that this also holds for $\langle S'_{p} \rangle$.
	Therefore, we obtain $\langle S_p \rangle \cap \langle S'_{p} \rangle=\{0\}$.
	Combining this with the fact that the groups $\langle S_p\rangle $ and $\langle S'_{p} \rangle$ together generate~\(\Gamma\), the claim follows.
\end{proof}

The following statement is a direct consequence of the proof of \cite[Theorems~6.4 and~6.6]{GRE24}.

\begin{lemma}\label{lemma:divisibilitylemma}
	Let $G \in \mH$ and assume that $G = \Cay(\Gamma, S)$ for an abelian group $\Gamma$ and a generating set $S$.
	\begin{enumerate}
		\item Suppose that $|\Gamma|$ is odd. If $S$ contains an element of order $pm$ for a prime $p$ and $m > 1$, then~$p$ divides $\kappa(G)$.
		\item If $|\Gamma|$ is even, then $\kappa(G)$ is even.
	\end{enumerate}
\end{lemma}

Last, many arguments that we use in this section only work for graphs that are not too small.
Therefore, the following result ensures that we do not have to take care of these cases.

\begin{lemma}\label{lemma:small-groups}
	Let $G$ be the Cayley graph of an abelian group \(\Gamma\).
	\begin{itemize}
		\item Assume that $n:=|\Gamma|\leq 16$. Then $G \in \mH$ if and only if one of the following holds
		\begin{itemize}
			\item $G \in \{K_n, C_n, C_4 \bx K_2\}$,
			\item $G=K_{k,k}$ where $k=n/2$,
			\item $G=C_k \bx K_2$ where $k=n/2$ is odd.
		\end{itemize}
		\item Assume that \(\Gamma \in \{\Z_3 \times \Z_3 \times \Z_3, \Z_9 \times \Z_3\}\). Then \(G \in \mH\) if and only if \(G = K_{27}\).
	\end{itemize}
\end{lemma}

We confirmed this lemma via exhaustive generation of all Cayley graphs using Sage. For most of the graphs, we could easily find two Hamiltonian cycles which could not be mapped onto one another. Further, we excluded some more graphs by counting all their Hamiltonian cycles and comparing this number to the size of the automorphism group. 
For the remaining graphs, a brute-force verification was performed to determine whether they lie in \(\mH\).

\subsection{Cyclic groups}

In this section, we characterize Hamiltonian-transitive Cayley graphs of cyclic groups, where the generating set contains an element whose order equals that of the group.
This later serves as a base for Cayley graphs of abelian groups.

\begin{theorem}\label{thm:cyclic_generator}
	Let~\(\Gamma \cong \Z_n\) and~\(S\) be a generating set of~\(\Gamma\) containing an element of order~\(n\). Let~\(G = \Cay(\Gamma,S)\). Then, the following are equivalent:
	\begin{enumerate}
		\item $ G \in \mH$, \label{item:cyclic_gen_graph_in_H}
		\item every Hamiltonian cycle~$C$ in~$G$ has Hamilton compression~$\kappa(C)=n$,  \label{item:cyclic_gen_rotation_one}
		\item we have $G \in \{K_n, C_n\}$ or $n=2m$ is even and $G=K_{m,m}$.\label{item:cyclic_gen_list_of_graphs_in_H}
	\end{enumerate}
\end{theorem}

\begin{proof} 
	We first show that~\eqref{item:cyclic_gen_list_of_graphs_in_H} implies~\eqref{item:cyclic_gen_graph_in_H}.
	Clearly, we have~\(K_n, C_n \in \mH\). 
	To see that also~\(K_{m,m}\in \mH\), let~\(A \sqcup B\) be the bipartition of~\(K_{m,m}\) and observe that every Hamiltonian cycle of~\(K_{m,m}\) alternately uses a vertex of~\(A\) and~\(B\). 
	Since every combination of permutations of~\(A\) and~\(B\) is an automorphism of~\(K_{m,m}\), 
	every Hamiltonian cycle can be mapped to every other Hamiltonian cycle by an automorphism of~\(K_{m,m}\).
	
	For the remaining implications, let us first argue that we can assume $1 \in S$: by assumption, there exists an element $x\in S$ of order $n$. Note that the group automorphism \(\varphi(kx):=k\) induces an isomorphism between $\Cay(\Gamma,S)$ and~$\Cay(\Gamma,\varphi(S))$, so we can assume w.l.o.g. that~$x=1$.
	
	Next, we show that~\eqref{item:cyclic_gen_graph_in_H} implies~\eqref{item:cyclic_gen_rotation_one}.
	Consider the Hamiltonian cycle $C=(0,1,\dots,\linebreak n-1)$ of $G$.
	Since, $y\mapsto y+1$ defines an automorphism of $G$, the cycle $C$ can be rotated by~1 so that the Hamilton compression of $C$ is $n$.
	Since $G\in \mH$, by \cref{lemma:samekappa} every cycle of~$G$ has Hamilton compression $n$.
	
	It remains to show that \eqref{item:cyclic_gen_rotation_one} implies~\eqref{item:cyclic_gen_list_of_graphs_in_H}. 
	Thus from now on, we assume that every Hamiltonian cycle in~$G$ has Hamilton compression~$n$ and we aim to show that~\(G\) is isomorphic to~\(K_n\),~\(C_n\), or~\(K_{m,m}\).
	If~\(G \not \cong C_n\), we have~\(k \in S\) for some~\(2 \leq k < n-1\).
	Consider the Hamiltonian cycle~\(C' = (0,k, k+1, k+2, \dots, n-1, k-1, k-2, \dots, 1)\) in~\(G\). 
	Since~$C'$ has Hamilton compression~$n$, rotating~$C'$ by any~$y\in \N$ yields an automorphism of~$G$. For every $y \in \{1, \dots, \min\{k-1,n-k-1\}\}$, this automorphism maps the edge~$\{k-1,k\}$ to~$\{k-1-y,k+y\}$. 
	The existence of this edge implies that $(k+y)- (k-1-y)=2y+1 \in S$. Hence~$\{1,3,5, \dots, \min\{2k-1, 2(n-k)-1\}\}\subseteq S$.
	
	For every $k\in S$, we have shown that $\{1,3, 5, \dots , \min\{2k-1, 2(n-k)-1\}\}\subseteq S$. Applying this repeatedly, we obtain that $S$ contains all odd numbers up to $n-1$.
	In particular, if~$n$ is odd, we obtain that $S=\{1,2,\dots, n-1\}$ because $S$ is inverse-closed. This means that~$G\cong K_n$.
	
	Now assume that~\(n\) is even.
	In this case, we have~\({A \coloneqq \{1, 3, 5, \ldots, n-1\} \subseteq S}\). If equality holds, we have~\(G \cong K_{m,m}\). 
	Otherwise,~\(G\) is a supergraph of~\(K_{m,m}\) and~\(S\) contains an element~\(2l\) for some~\(l < \frac{n}{2}\).
	We show that, for every~\(l'\leq \frac{n}{2}\), we also have~\(2l'\in S\): 
	Since~\(S\) contains all elements in~$A$, note that every permutation of~\((0,1,2,\dots, n-1)\), where precisely every second element is in~$A$, is a Hamiltonian cycle of~\(G\).
	Let $C'$ be the Hamiltonian cycle obtained from $(0,1,2, \dots, n-1)$ by replacing $2l$ and $2l'$.
	Rotating~$C'$ by~$n-1$ (i.e., by 1 in the other direction) maps the edge~$\{1,2l+1\}$ to~$\{0,2l'\}$, and hence,~\(2l'\in S\). 
	This shows that~\(S=\Gamma\setminus \{0\}\), i.e.,~\(G\cong K_n\).
\end{proof}

\subsection{Layer structure theorem}\label{sec:layer-structure}

In this subsection, we investigate a general framework for graphs containing a grid-like structure and having non-trivial Hamilton compression.
As we explain later, Cayley graphs of abelian groups fulfill these conditions 
and thus fall within the scope of this subsection.
Note that despite the fact that Cartesian products do have a layer structure, they do not necessarily have non-trivial Hamilton compression, so the results from this section cannot be applied in that case.

We begin by formalizing the grid-like structure which we work with.

\begin{definition}[Layer structure]\label{def:layer-structure}
	Let $G$ be a graph and $l \in \N_{\geq 1}$.
	If there is a decomposition $V(G) = V_1 \dcup \dots \dcup V_l$ such that
	\begin{enumerate}
		\item for $i \in \{1, \dots, l-1\}$, $G$ contains a perfect matching between $V_i$ and $V_{i+1}$ that defines an isomorphism between $G[V_i]$ and $G[V_{i+1}]$, and
		\item the graph $K:= G[V_1] \cong \dots \cong G[V_l]$ is Hamiltonian,
	\end{enumerate}
	we say that~$G$ has~\emph{$K$-$l$-layer structure}.
	In this case, we call $G[V_1], \dots ,G[V_l]$ the \emph{layers} of~$G$.
\end{definition}

Note that, if $G$ has $K$-$l$-layer structure, then $G$ has a spanning subgraph isomorphic to~$K \bx P_l$.
Given some ordering of the vertices of $K$ (in the following usually given by a Hamiltonian cycle of $K$), by $v_{j,i}$ we denote the $i$-th vertex in the $j$-th layer of $G$.
Next, we construct a class of Hamiltonian cycles that every graph with layer structure contains.
Since we will only use this construction for graphs of odd order (in particular, $l$ will always be odd), we restrict ourselves to this case.
However, note that this construction can be easily generalized to graphs of even order (but some of the steps later throughout the proofs cannot be generalized for even order).

\begin{definition}[Zigzag cycle]\label{def:zigzag}
	Assume that $G$ has $K$-$l$-layer structure with $l\geq 5$ odd and set $k := |K|$. Let~$C_K = (v_0, \dots, v_{k-1})$ be a Hamiltonian cycle of $K$.
	Given a tuple~$\vec{a} = (a_1, a_2, \dots, a_{(l-3)/2}) \in \{0, \dots, k-2\}^{(l-3)/2}$, we define the \emph{$C_K$-$\vec{a}$-zigzag cycle} as the Hamiltonian cycle illustrated in \cref{fig:zigzag}.
	For a zigzag cycle $C$, we define $\mu(C)$ to be the number of edges in the segment between \(v_{0,k-2}\) and \(v_{l-1,k-1}\) that contains vertex $v_{1,0}$, i.e., the length of the segment consisting of the green, pink, and orange line segments in \cref{fig:zigzag}.
\end{definition}

\begin{figure}
	\centering
	\begin{tikzpicture}[scale=0.8, every node/.style={circle, draw, inner sep=0pt, minimum size=1.8mm}]	
		\foreach \x in {0,1,2,3,4,5} {
			\foreach \y in {0,1,2,3,4,5,6} {
				\node (\x\y) at (\x,\y) {};
			}
		}
		\node [fill] at (4,6) {};
		\node [fill] at (5,0) {};
		
		\foreach \x in {0,1,2,3,4,5} {
			\foreach \y in {0,1,2,3, 4,5} {
				\draw[gray] (\x\y) to (\x\the\numexpr\y+1\relax);
			}
		}
		
		\foreach \x in {0,1,2,3,4} {
			\foreach \y in {0,1,2,3,4,5,6} {
				\draw[gray] (\x\y) to (\the\numexpr\x+1\relax\y);
			}
		}
		
		\draw[decorate,decoration={brace,amplitude=10pt}] 
		(0,6.5) -- (5,6.5) node[midway,yshift=0.7cm, draw=none] {\(k\)};
		\draw[decorate,decoration={brace,amplitude=5pt}] 
		(0.06,4.16) -- (2.94,4.16) node[midway,yshift=3mm, draw=none] {\footnotesize{$a_1$}};
		\draw[decorate,decoration={brace,amplitude=5pt}] 
		(0.06,2.16) -- (1.94,2.16) node[midway,yshift=3mm, draw=none] {\footnotesize{$a_2$}};
		\node [draw=none] at (-1.2,6) {\footnotesize{layer 0}};
		\node [draw=none] at (-1.2,5) {\footnotesize{layer 1}};
		\node [draw=none] at (-1.3,0) {\footnotesize{layer $l-1$}};
		
		\draw[line width=2pt, green!70!black] (46) to (45) to (35) to (25) to (15) to (05) to (04);
		\draw[line width=2pt, magenta!90!black] (04) to (14) to (24) to (34) to (33);
		\draw[line width=2pt, orange!90!black] (33) to (23) to (13) to (03) to (02);
		\draw[line width=2pt, magenta!90!black] (02) to (12) to (22) to (21);
		\draw[line width=2pt, orange!90!black] (21) to (11) to (01) to (00);
		\draw[line width=2pt, green!70!black] (00) to (10) to (20) to (30) to (40) to (50);
		\draw[very thick, blue] (50) to (51) to (41) to (31) to (32) to (42) to (52) to (53) to (43) to (44) to (54) to (55) to (56) to [bend right=15] (06);
		\draw[very thick, blue] (06) to (16) to (26) to (36) to (46);	
	\end{tikzpicture}
	\caption{Zigzag cycle for $l=7$ odd and $\vec{a}=(3,2)$. The filled vertices are the endpoints of the described segment of length~$\mu(C)$.}
	\label{fig:zigzag}
\end{figure}
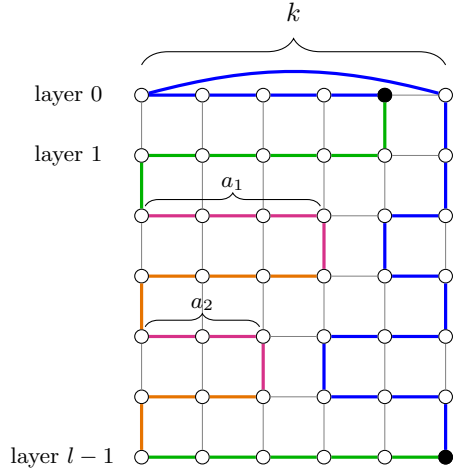

The following lemma is a direct consequence of the definition and shows that, by suitably choosing $\vec{a}$, many desirable values for $\mu(C)$ for a zigzag cycle $C$ can be achieved.

\begin{restatable}{lemma}{lemzigzag}\label{lem:zigzag_mu}
	Let $G$ be a graph with $K$-$l$-layer structure, where $l\geq 5$ is odd, let $k:=|K|$, and let $C$ be the $C_K$-$\vec{a}$-zigzag cycle for some choice of $C_K$ and $\vec{a}$. Then we have
	\begin{equation*}
		\mu(C)=2k-1+2\sum_{i=1}^{(l-3)/2} (a_i+1).
	\end{equation*}
	
	In particular, we obtain that $\mu(C)$ is odd. 
	Even further, for any odd number $\mu^*\in \N$ with~$2k-1+(l-3)\leq \mu^* \leq 2k-1+(l-3)(k-1)$, there is a choice for $\vec{a}$ such that the~$C_K$-$\vec{a}$-zigzag cycle $C$ fulfills~$\mu(C)=\mu^*$.
\end{restatable}
\begin{proof}
	Note that the length of the green segment at the top is $k$ and the length of the green segment at the bottom is $k-1$.
	Moreover, the length of each pink segment is $a_i+1$ and the length of the following orange segment is also $a_i+1$.
	Together, we obtain that $\mu(C)$ is as stated.
	The given bounds follow from $0\leq a_i \leq k-2$ for all~$i \in \{1, \dots ,(l-3)/2\}$ and from the observation that, increasing one of the $a_i$ by 1, increases the value of $\mu(C)$ by 2.
\end{proof}

Now, we have all the prerequisites in place to prove the main result of this subsection.

\begin{restatable}{theorem}{layerstructuretheorem}\label{thm:layerstructure}
	Let $G \in \mH$ be a graph of odd order that has $K$-$l$-layer structure and let~$k:=|K|$.
	Assume that one of the following conditions holds
	\begin{enumerate}
		\item $k,l\geq 7$ and $\kappa(G)\geq 5$
		\item $k,l \geq 5$ and, for every prime $p$ with $p|n:=|G|$, we have $p|\kappa(G)$.
	\end{enumerate}
	Then~$K$ is isomorphic to $C_k$ or $K_k$.
\end{restatable}

\begin{proof}
	The strategy is to show that, for every Hamiltonian cycle $C_K$ of $K$, we have \hbox{$\kappa(C_K) = k$}. 
	By~\cite[Section~1.4]{GRE24}, this implies that $K$ is the Cayley graph of a cyclic group and the generating set contains an element of order $k$.
	The assertion then follows from Theorem~\ref{thm:cyclic_generator}.
	
	Fix a Hamiltonian cycle $C_K$ of $K$.
	To prove that $\kappa(C_K) = k$, we first show the following claim.
	\begin{restatable}{claim}{zigzagclaim}\label{claim:zigzag_multiple}
		There is a choice for~$\vec{a}$ such that the~$C_K$-$\vec{a}$-zigzgag cycle $C$ fulfills that $\mu(C)$ is a multiple of $\alpha:=n/\kappa(G)$.
	\end{restatable}
	\begin{proof}[Proof of \cref{claim:zigzag_multiple}]
		Since $n$ is odd, $\alpha$ is odd as well.
		By \cref{lem:zigzag_mu}, it suffices to show that there is some odd $m \in \N$ such that $m \alpha \in [\mu_{\text{min}}, \mu_{\text{max}}]$, where
		$\mu_{\text{min}}:=2k-1+l-3$ and~$\mu_{\text{max}}=2k-1+(l-3)(k-1)=(l-1)(k-1)+1$.
		
		First, assume that condition (1) in the Theorem statement is fulfilled.
		Note that then $\alpha=n/\kappa(G)\leq n/5$. We have
		\begin{align*}
			\mu_{\text{max}}-\mu_{\text{min}}=(l-3)\cdot (k-2)
			\geq \frac{4}{7}l \cdot \frac{5}{7}k
			= \frac{20}{49} n \geq \frac{2}{5} n \geq 2 \alpha.
		\end{align*}
		In other words, the interval $[\mu_{\text{min}}, \mu_{\text{max}}]$ contains at least two multiples of $\alpha$, in particular an odd multiple, which completes the proof in case that condition (1) holds.
		
		Now, assume that condition (2) holds, but (1) does not.
		This means that one of $k,l$ equals~5, or that $\kappa(G)=3$ (note that $k,l,\kappa(G)$ are odd and $\kappa(G)=1$ is not possible, because there is some prime number that divides $\kappa(G)$).
		We begin with the case that one of $k,l$ equals 5.
		Then $5|n$ so that~$\kappa(G)=5m'$ for some odd $m'\in \N$ such that $\alpha m' = n/5$.
		We show that for $m=3m'$, we have $m\alpha=3m'\alpha=3n/5=3kl/5 \in [\mu_{\text{min}},\mu_{\text{max}}]$. For this, note that	
		\begin{align*}
			\mu_{\text{min}}=2k+l-4 \leq 3 \max(k,l) \leq \frac{3}{5} kl
			\leq \frac{4}{5} k \cdot \frac{4}{5} l
			\leq (k-1)(l-1)+1=\mu_{\text{max}}.
		\end{align*}
		Now, it only remains to investigate the case in (2) where $\kappa(G)=3$.
		Since every prime number that divides $n=kl$ also divides $\kappa(G)$, note that $\kappa(G)=3$ is only possible if $k$ and $l$ are powers of~3.
		In particular, $k,l\geq 9$.
		If $k=l=9$, we have $\alpha=27$, $\mu_{\text{min}}=23, \mu_{\text{max}}=65$, so that the assertion holds with $m=1$.
		Therefore, assume that one of $k,l$ is greater than 9, i.e., at least~27.
		Note that then
		\begin{equation}\label{eq:mumax}
			\mu_{\text{max}}\geq (l-3)(k-1)\geq \frac{6}{9}l \cdot \frac{7}{9} k \geq \frac{1}{3} kl = \frac{1}{3} n \geq \alpha
		\end{equation}
		and	
		\begin{align}\label{eq:3mumin}
			3\cdot \mu_{\text{min}}&=3(2k-1+l-3)=6k+3l-12\leq 6k+3l\nonumber\\ 
			&\leq  9 \max(k,l) - 3\left(\max(k,l)-\min(k,l)\right) \leq kl-k-l+1=\mu_{\text{max}},
		\end{align}
		where the last inequality holds for $\max(k,l)\geq 27$ and $\min(k,l)\geq 9$.
		The inequalities \eqref{eq:mumax} and \eqref{eq:3mumin} together show that one can choose $m$ to be a power of~3.
		This completes the proof of the claim. 
	\end{proof}
	Since $G\in \mH$, every Hamiltonian cycle of $G$, in particular~$C$, is $\kappa(G)$-symmetric.
	Therefore, rotating $C$ by a multiple of~$\alpha$ defines an automorphism of~$G$.
	Together with \cref{claim:zigzag_multiple}, this means that~$C$ can be rotated by $\mu(C)$ (in counter-clockwise direction in \cref{fig:zigzag}).
	This maps the first layer of $G$ bijectively to the last, inducing an automorphism of $K$.
	More precisely, $v_{0,j}$ is mapped to~$v_{l-1,j+1 \bmod k}$. 
	The induced automorphism of $K$ then maps $v_j$ to~$v_{j+1 \bmod k}$, i.e., it rotates~$C_K$ by 1. Therefore, we have $\kappa(C_K)=k$.
\end{proof}


\subsection{Cayley graphs of odd order}
\label{subsec:cayleyodd}

In this subsection, we prove our main result. We begin by restating it.

\thmoddcayleygraphs*

\label{sec:group-induced-layer-structure}
As a first step, let us argue that Cayley graphs of abelian groups have a natural layer structure, which allows us to use results from the previous subsection.

\begin{observation}\label{obs:group-induced-layer}
	Let $G$ be the Cayley graph of an abelian group $\Gamma$ with generating set $S$.
	Assume there is~$S'\subseteq S$ such that $\langle S' \rangle \cap S = S'$ and $|\langle S' \rangle|\geq 3$.
	Let $\Delta:=\langle S' \rangle$, $K:=\Cay(\Delta, S')$ and~$l:=|\Gamma:\Delta|$.
	Then $G$ has $K$-$l$-layer structure (\cref{def:layer-structure}), where the layers are given by the cosets of $\Delta$.
	We call such a layer structure \textit{group-induced}.
\end{observation}

\begin{proof}
	To see this, note that the cosets of $\Delta$, given by~$\{\gamma+\Delta: \gamma \in \Gamma\}$, provide a partition of the vertices of $G$ into $l$ sets.
	Since $\langle S' \rangle \cap S = S'$, for each of these sets, the induced subgraph is~$G[\gamma+\Delta]\cong\Cay(\Delta, S')=K$.
	Moreover, since $K$ is the Cayley graph of an abelian group on at least~3 vertices, it is Hamiltionian.
	Similarly, $G':=\Cay(\Gamma/\Delta, S \setminus S')$ is the Cayley graph of an abelian group and therefore traceable, i.e., it contains a Hamiltonian path~$(\gamma_1+\Delta, \dots, \gamma_l+\Delta)$.
	This means that, for $i \in \{1, \dots, l-1\}$, there exists some $s_i \in S\setminus S'$ such that $s_i+\gamma_i+\Delta=\gamma_{i+1}+\Delta$.
	Then the edges of $G$ corresponding to~$s_i$ form a perfect matching between $\gamma_i +\Delta$ and~$\gamma_{i+1} +\Delta$. Moreover, since $x \mapsto x+s_i$ is an automorphism of~$G$, the matching defines an isomorphism between $\gamma_i +\Delta$ and $\gamma_{i+1} +\Delta$.
	Therefore,~$G$ has~$K$-$l$-layer structure as stated.
\end{proof}

In the next lemma, we make use of this structure and our results from the previous~section.

\begin{restatable}{lemma}{lemlayercayley}\label{lemma:layercayley}
	Let $\Gamma$ be an abelian group of odd order with generating set $S$. Suppose that the Cayley graph $G = \Cay(\Gamma, S)$ is Hamiltonian-transitive and $\kappa(G) > 1$.
	Assume that $G$ has a group-induced $K$-$l$-layer structure with $|K|,l\geq 5$.
	Then $K$ is complete or a cycle.
\end{restatable}

\begin{proof}
	Let $S'$ and $\Delta$ be as described in \cref{obs:group-induced-layer} corresponding to the $K$-$l$-layer structure with $|K|,l\geq 5$.
	Our strategy is to show that condition (2) in \cref{thm:layerstructure} is fulfilled, which then immediately implies the assertion. For this, it is only left to prove that, for every prime number $p$ that divides $|G|$, $p$ also divides~$\kappa(G)$.
	
	First, assume that $S$ contains an element of order $pm$ for some $m > 1$.
	By Lemma~\ref{lemma:divisibilitylemma}, $p$ then divides $\kappa(G)$.
	Therefore, assume from now on that
	the order of every element in $S$ is either $p$ or coprime to $p$.
	Let $S_p$ contain the elements of $S$ of order~$p$ and $S_{p'}$ the elements of order coprime to $p$.
	By Lemma~\ref{lemma:compositeorder}, we obtain~$\Gamma =  \langle S_p \rangle \times \langle S_{p'} \rangle$ and 
	by Lemma~\ref{lemma:splittingcayley}, it follows~$G = \Cay(\langle S_p \rangle, S_p) \bx \Cay(\langle S_{p'} \rangle, S_{p'})$.
	
	Note that $| \langle S_p \rangle |$ is a power of $p$ (if $| \langle S_p \rangle |$ contains some other prime factor $q\neq p$, $\langle S_p \rangle$ contains an element of order $q$, which is a contradiction). 
	Similarly,  $| \langle S_{p'} \rangle |$ is coprime to $p$.
	Therefore, $\Cay(\langle S_p \rangle, S_p)$ and $\Cay(\langle S_{p'} \rangle, S_{p'})$ cannot contain the same factor in a Cartesian product decomposition, i.e., they are relatively prime.
	In addition, since they are Cayley graphs of abelian groups, they are Hamiltonian.
	Since~$G = \Cay(\langle S_p \rangle, S_p) \bx \Cay(\langle S_{p'} \rangle, S_{p'})$ is Hamiltonian-transitive, Lemma~\ref{lem:prime-not-in-H} implies that one of the factors is trivial.
	
	Since we assumed that $p$ divides $|G|$, $S_p$ cannot be trivial, i.e., we have $S=S_p$.
	We obtain that~$|\Gamma|=|\langle S_p \rangle |$ is a power of $p$.
	Then~$p$ divides $\kappa(G)$ because~$\kappa(G)$ divides~$|G|=|\Gamma|$ and we have $\kappa(G) > 1$ by assumption.
	Therefore, condition (2) in  \cref{thm:layerstructure} is fulfilled, which completes the proof of the lemma.
\end{proof}

In the following two results, we build on \cref{lemma:layercayley} and further resolve the cases where~$K$ is complete or a cycle.
Note in the next lemma that we do not require the graph to be of odd order, which will allow us to reuse the statement for groups of even order.

\begin{restatable}{lemma}{lemcayleycomplete}\label{lemma:cayleycomplete}
	Let $\Gamma$ be an abelian group with generating set $S$. 
	Suppose that the Cayley graph~$G = \Cay(\Gamma,S)$ is Hamiltonian-transitive and  $\kappa(G) > 1$.
	Assume that $G$ has a group-induced $K$-$l$-layer structure where $K$ is a complete graph with $|K|\geq 5$.
	Then $G$ is complete.
\end{restatable}

\begin{proof}
	Let $S'$ and $\Delta:=\langle S' \rangle$ be as described in \cref{obs:group-induced-layer} giving the layer structure as in the assumption of the lemma.
	Recall that the layers are the cosets of~$\Delta$, i.e., the elements of the group $\Gamma':=\Gamma/\Delta$ and there is an edge between two layers if there is an edge in the graph $G':=\Cay(\Gamma', S\setminus S')$ between the corresponding vertices.
	
	As a first step, we argue that, if there exists an edge between two layers, all edges between the two layers are present.
	So fix two neighboring layers $L_1, L_2 \in G'$ and assume they contain vertices (of $G$) $u\in L_1$ and $v\in L_2$.
	We show that $u$ and $v$ are adjacent by arguing that there exists a Hamiltonian cycle in~$G$ in which $u$ and $v$ are mapped to the same layer when rotating the cycle by some multiple of \(\alpha:=n/\kappa (G)\). 
	
	To construct such a cycle, note that, since layers $L_1$ and $L_2$ are neighboring and each layer is complete, there exists a vertex $w$ in layer $L_1$ that is adjacent to both $u$ and $v$.
	The cyles that we construct will all contain the segment $(v,w,u)$.
	Moreover, since~$G'$ is a Cayley graph of an abelian group, $G'$ contains a Hamiltonian path starting with the edge~$\{L_1, L_2\}$ (see~\cite{CHE81}), i.e., we can assume that $L_1$ is the first layer and $L_2$ is the second layer in the group-induced layer structure.
	To construct the desired Hamiltonian cycle, we distinguish the cases \(l=2\) and \(l\geq 3\). We begin with~\(l=2\). In this case, by~\cref{lemma:divisibilitylemma}, \(\kappa(G)\) is even and hence, a rotation of any cycle by \(n/2=k\) is possible. Thus, observe that the cycle depicted in~\cref{fig:l=2tocomplete} has the desired property. 
	
	Next, we consider the case \(l\geq 3\). Note that, since $\kappa(G)>1$, we have $\alpha\leq n/2=lk/2\leq (l-1)(k-1)$, where we have used~$l\geq 3$ and~$k\geq 5$.
	In particular, every cycle can be rotated by some value~\(x\) with $x\leq (l-1)(k-1)$ and $x\geq (l-1)(k-1)/2\geq l+1$ (where we have again used that $l\geq 3, k\geq 5$).
	We can construct a Hamiltonian cycle as follows (depicted in \cref{fig:l=3tocomplete} for $l=3$): Start with the segment $(v,w,u)$. 
	Then, in each of the layers 1 to $l-2$, visit a suitable number of vertices between 0 and $k-1$. In the final layer, visit all vertices consecutively, and then extend the path to a valid Hamiltonian cycle.
	It is easy to see that the number of vertices visited in the middle layers can be chosen such that $u$ and $v$ are both mapped to the last layer.
	
	So far we have established that, if two layers of $G'$ are adjacent, all of their vertices in $G$ are adjacent.
	It remains to show that $G'$ is complete.
	If $l=|G'|\leq 2$, there is nothing to do, so assume $l\geq 3$.
	Since $G'$ is the Cayley graph of an abelian group on at least~3 vertices, it is Hamiltonian.
	Number the layers such that $(L_1, \dots, L_l)$ is a Hamiltonian cycle of the layers.
	Since we have already established that all vertices of $G$ in consecutive layers are adjacent, note that the following is a Hamiltonian cycle: $C:=(v_{0,0}, \dots, v_{0,l-1}, v_{1,0}, \dots, v_{1,l-1}, \dots, v_{k-1,0}, \dots, v_{k-1,l-1})$.
	Observe that this cycle has the property that every pair of vertices that have distance $l$ along the cycle, is adjacent in~$G$.
	Now, we construct another cycle $C'$ that does not have this property.
	Since $G'$ is not complete, there are two layers $L', L''$ which are not connected.
	Starting in layer $L'$, follow the Hamiltonian cycle of $G'$ in the direction in which $L''$ is reached only after at least $l/2$ steps.
	By using between one and two vertices from each layer along this route, we obtain a Hamiltonian cycle of \(G\) in which two vertices from layers \(L'\) and \(L''\) have distance exactly~\(l\).
	Since they are not adjacent by assumption, this cycle cannot be mapped to cycle~$C$ contradicting that~$G$ is Hamiltonian-transitive.
	Hence, $G'$ is complete, which completes the proof that $G$ is complete.
\end{proof}

\begin{figure}
\centering
\begin{tikzpicture}[scale=1, every node/.style={circle, draw, inner sep=0pt, minimum size=3.5mm}]
	\foreach \x in {0,1,2,3,4,5,6} {
		\node[style=vertex] (v\x) at (\x,1) {};
		\node[style=vertex] (u\x) at (\x,0) {};
	}
	\foreach \x in {0,1,2,3,4,5,6} {
		\draw[gray] (v\x) to (u\x);
	}
	\node[fill=white] (v1) at (1,1) {\footnotesize{$v$}};
	\node[fill=white] (u0) at (0,0) {\footnotesize{$u$}};
	\node[fill=white] (w) at (0,1) {\footnotesize{$w$}};
	
	\draw[blue, very thick] (v1) to (w) to (u0) to [bend right=20] (u3)  (u3) to (u4) to (u5) to (u6) to (v6) to (v5) to (v4) to (v3) to (v2) to (u2) to (u1) to (v1);
	
	\draw[gray] (v1) to (v2);
	\draw[gray] (u0) to (u1);
	\draw[gray] (u2) to (u3);
	\node[circle, draw, fill=black, inner sep=0pt, minimum size=2mm] at (6,1) {};
	\node[circle, draw, fill=black, inner sep=0pt, minimum size=2mm] at (4,1) {};
\end{tikzpicture}
\caption{Hamiltonian cycle in the proof of Lemma~\ref{lemma:cayleycomplete} for \(l=2\) and \(k\geq 5\), depicted for $k = 7$. A rotation of the cycle by \(k=n/2\) maps \(u\) and \(v\) to the two filled vertices in the upper layer.}
\label{fig:l=2tocomplete}
\end{figure}

\begin{figure}
\centering
\begin{tikzpicture}[scale=1, every node/.style={circle, draw, inner sep=0pt, minimum size=3.5mm}]
	\foreach \x in {0,1,2,3,4} {
		\node[style=vertex] (v\x) at (\x,1) {};
		\node[style=vertex] (u\x) at (\x,0) {};
		\node[style=vertex] (w\x) at (\x,-1) {};
	}
	\node[fill=white] (v1) at (1,1) {\footnotesize{$v$}};
	\node[fill=white] (u0) at (0,0) {\footnotesize{$u$}};
	\node[fill=white] (v0) at (0,1) {\footnotesize{$w$}};
	
	\foreach \x in {0,1,2,3,4} {
		\draw[gray] (u\x) to (w\x);
		\draw[gray] (v\x) to (u\x);
	}
	\draw[blue, very thick] (v0) to (u0);
	\draw[LimeGreen, line width=2pt] (u0) to (w0);
	\draw[blue, very thick] (w0) to (w1) to (w2) to (w3) to (w4);
	\draw[blue, very thick] (w4) to (u4);
	\draw[blue, very thick] (u4) to [bend right = 30] (u1);
	\draw[blue, very thick] (u1) to (u2) to (u3) to (v3) to (v4);
	\draw[blue, very thick] (v4) to [bend right = 30] (v2);
	\draw[blue, very thick] (v2) to (v1) to (v0);
	\node[circle, draw, fill=black, inner sep=0pt, minimum size=2mm] at (1,-1) {};
	\node[circle, draw, fill=black, inner sep=0pt, minimum size=2mm] at (3,-1) {};
\end{tikzpicture}
\hspace{1cm}
\begin{tikzpicture}[scale=1, every node/.style={circle, draw, inner sep=0pt, minimum size=3.5mm}]
	\foreach \x in {0,1,2,3,4} {
		\node[style=vertex] (v\x) at (\x,1) {};
		\node[style=vertex] (u\x) at (\x,0) {};
		\node[style=vertex] (w\x) at (\x,-1) {};
	}
	\node[fill=white] (v1) at (1,1) {\footnotesize{$v$}};
	\node[fill=white] (u0) at (0,0) {\footnotesize{$u$}};
	\node[fill=white] (v0) at (0,1) {\footnotesize{$w$}};
	\node[circle, draw, fill=black, inner sep=0pt, minimum size=2mm] at (0,-1) {};
	\node[circle, draw, fill=black, inner sep=0pt, minimum size=2mm] at (2,-1) {};
	\foreach \x in {0,1,2,3,4} {
		\draw[gray] (u\x) to (w\x);
		\draw[gray] (v\x) to (u\x);
	}
	\draw[blue, very thick] (v0) to (u0);
	\draw[LimeGreen, line width=2pt] (u0) to (u1) to (u2) to (u3) to (w3);
	\draw[blue, very thick] (w3) to (w2) to (w1) to (w0);
	\draw[blue, very thick] (w0) to [bend left = 20] (w4);
	\draw[blue, very thick] (w4) to (u4) to (v4) to (v3) to (v2) to (v1) to (v0);
\end{tikzpicture}
\caption{Hamiltonian cycles in the proof of Lemma~\ref{lemma:cayleycomplete} for \(l=3\) depicted for $k = 5$ and two different values of $x\in\{k-1, \dots, 2k-2\}$.
	The cycles are constructed such that the length of the green thick segment is $x-k+2\in\{1, \dots, k\}$.
	We have used for the construction that each layer is a complete graph. Observe that rotation by $x$ maps $u$ to the second-to-last vertex used by the cycle in the last layer, and $v$ to the fourth-to-last vertex (marked by filled vertices).
}
\label{fig:l=3tocomplete}
\end{figure}

Next, we resolve the case where each layer is a cycle.

\begin{restatable}{lemma}{lemcayleycycle}\label{lemma:cayleycycle}
Let $\Gamma$ be an abelian group of odd order with generating set $S$. 
Suppose that the Cayley graph $G = \Cay(\Gamma,S)$ is Hamiltonian-transitive and  $\kappa(G) > 1$. Assume that $G$ has a group-induced $K$-$l$-layer structure with $k\coloneqq|K|\geq 5$ and $l\geq5$. Then we have $K \not \cong C_k$. 
\end{restatable}

\begin{proof}
Assume otherwise.
We use a suitable zigzag-cycle of $G$ in the group-induced layer structure to prove this statement.
As these are only defined for odd~$l$, we first naturally extend the definition for even $l$ as follows:  Take a zigzag cycle for~$l-1$ (see \cref{fig:zigzag}) and ``replace'' the edge between $\{v_{0,0},v_{0,k-1}\}$ by visiting all edges of an additional first layer.

It is easy to see that the parameters of a zigzag-cycle can be chosen such that the vertices of the last layer $v_{l-1,0}$ and $v_{l-1,k-1}$ are mapped to different layers when rotating by $\alpha=n/\kappa(G)$: In fact, they can only be mapped to the same layer if the entire last layer is mapped to the first layer. Modifying any of the parameters of the zigzag-cycle by one prevents this.

If $v_{l-1,0}$ and $v_{l-1,k-1}$ are mapped to different layers, there exist vertices in the last layer at distance 3 that are mapped to adjacent vertices.
Since $G$ is Hamiltonian-transitive, this contradicts the last layer being a cycle.
\end{proof}

The last step before concluding the main theorem is to settle the cases where $|K|$ or $l$ is small.

\begin{restatable}{lemma}{lemcayleythreelayers}\label{lemma:cayleythreelayers}
Let $\Gamma$ be an abelian group of odd order with generating set $S$. 
Suppose that the Cayley graph $G = \Cay(\Gamma,S)$ is Hamiltonian-transitive and  $\kappa(G) > 1$.
Assume that, for all~$s\in S$, $\operatorname{ord}(s)<5$ or $|\Gamma : \langle s\rangle|<5$.
Then~$\Gamma=\Z_3^d$ for $d\geq4$ or $G$ is complete or a~cycle.
\end{restatable}

\begin{proof}
If there is an element $s \in S$ with $\langle s \rangle=\Gamma$, $G$ is complete or a cycle by $\cref{thm:cyclic_generator}$. 
Since $\Gamma$ is of odd order, we can thus assume for every $s \in S$ that $\operatorname{ord}(s)=3$ or $|\Gamma : \langle s \rangle|=3$.

If we have $\operatorname{ord}(s)=3$ for every $s\in S$, then $\Gamma \cong \Z_3^d$ for some $d \geq 1$ (by the fundamental theorem of finitely generated abelian groups). By~\cref{lemma:small-groups}, $d\leq 3$ is only possible if $G$ is complete, so the statement holds.
Therefore, we can assume there exists some $s\in S$ with  $\operatorname{ord}(s)\geq 5$ and $|\Gamma : \langle s \rangle|=3$.

Consider the group-induced $K$-$l$-layer structure described in \cref{obs:group-induced-layer} for~$\Delta\coloneqq\langle s \rangle$ and let $k\coloneqq|K|$. Let $t \in S \setminus \Delta$. 
First, we argue that $\operatorname{ord}(t)\geq 5$ (and therefore $|\Gamma : \langle t \rangle|=3)$.
Assume otherwise, i.e., assume $|\langle t \rangle| = 3$.
As $|\Gamma : \Delta| = 3$, $t + \Delta$ generates~$\Gamma/\Delta$ and hence~$\Gamma = \Delta + \langle t\rangle$.
Since $|\langle t \rangle| = 3$, we even have $\Gamma = \Delta \times \langle t\rangle$ and hence $G = K \bx K_3$. 
Decompose $K = Q_1^{\bx r_1} \bx \dots \bx Q_r^{\bx r_n}$ into its prime factors. It is easy to see that $Q_1^{r_1}, \dots, Q_r^{r_n}$ are Cayley graphs of subgroups of $\Gamma$.
In particular, they have odd order at least 3 and are Hamiltonian. 
By Theorem~\ref{thm:cart-prod} and Theorem~\ref{thm:cartesian-product-K2}, we obtain a contradiction. 
Thus, we obtain $|\langle t \rangle| \geq 5$ and $|\Gamma : \langle t \rangle| = 3$.

In particular, we have $3 = |\Gamma : \langle t \rangle| = |(\Delta + \langle t \rangle) : \langle t \rangle| = |\Delta : \Delta \cap \langle t \rangle|$ by the isomorphism theorem. 
Hence $k = \operatorname{ord}(s)=|\Delta|$ is a multiple of $3$.

Consider the Hamiltonian cycle $C$ depicted in Figure~\ref{fig:graphenl=3_kreis}.
As in the proof of Lemma~\ref{lemma:layercayley}, we show that $\kappa(G)$ is divisible by $3$, and hence rotating $C$ by $k$ is an automorphism of~$G$. 
It maps $\{v_2, u_2\}$ to $\{w_0, v_{k-1}\}$, and hence $w_0$ and $v_{k-1}$ are adjacent in $G$.
Hence there is a cycle $C' = (w_0, u_0, v_0, v_1, u_1, w_1, w_2, u_2, v_2, v_3, \dots, v_{k-1})$.
Note that the vertices~$\{u_0, u_1\}, \{u_1, u_2\}, \{u_2, u_3\}, \dots, \{u_{k-2}, u_{k-1}\}, \{u_{k-1}, u_0\}$ are pairs of adjacent vertices of distance~$3$ on~$C'$. 
Observe that mapping $C'$ to $C$, one of these pairs will be mapped to the same layer: Every third vertex in $C'$ is in such a pair, and the layers have size at least 5. If such a pair is not contained in the last layer, we have $k=5$ and one of the $u_i$ is mapped to $w_2$. But then, one can easily see that such a pair exists in the second layer.

Having two adjacent vertices at distance 3 in the same layer implies that $3s \in S$. 
As $3$ divides $k=\operatorname{ord}(s)$, we have $\operatorname{ord}(3s) = k/3$ and we have  $|\Gamma : \langle 3s\rangle|= 3\cdot  |\Gamma : \langle s\rangle|=9$.
If $k > 9$, we also have $\operatorname{ord}(3s) > 3$. 
This is a contradiction to the assumption of the lemma as we have $3s\in S$.		
The only case left is $k = 9$ (so $|\Gamma| = 27$), which can be checked directly (see Lemma~\ref{lemma:small-groups}).
\end{proof}

\begin{figure}
	\centering
	\begin{tikzpicture}[scale=1, every node/.style={circle, draw, inner sep=0pt, minimum size=3.5mm}]
		\foreach \x in {0,1,2,3,4,5,6,7} {
			\node (v\x) at (\x,1) {\footnotesize{$v_{\x}$}};
			\node (u\x) at (\x,0) {\footnotesize{$u_{\x}$}};
			\node (w\x) at (\x,-1) {\footnotesize{$w_{\x}$}};
		}
		\foreach \x in {0,1,2,3,4,5,6,7} {
			\draw[gray] (u\x) to (w\x);
		}
		\foreach \x in {0,1,2,3,5,6,7} {
			\draw[gray] (v\x) to (u\x);
		}
		\draw[blue, very thick] (v0) to (u0) to (w0) to (w1) to (w2) to (w3) to (w4) to (w5) to (w6) to (w7) to (u7) to (u6) to (u5) to (u4) to (u3) to (u2) to (u1) to (v1) to (v2) to (v3) to (v4) to (v5) to (v6) to (v7) to [bend right = 20] (v0);
		\draw[magenta, very thick] (v4) to (u4);
		\draw[magenta, very thick, dashed] (u0) to (w6);
	\end{tikzpicture}
	\caption{The Hamiltonian cycle in the proof of \cref{lemma:l=3}, depicted for $k = 8$.}
	\label{fig:graphenl=3_kreis}
\end{figure}

Now, we have all the tools at hand to prove \cref{thm:odd-cayley-graphs}.

\begin{proof}[Proof of \cref{thm:odd-cayley-graphs}]
First, assume that $\kappa(G) = 1$. 
Then by~\cite[Theorem~6.4]{GRE24}, this implies~\hbox{$S = \{s_{p_1}, \dots, s_{p_r}\}$} for pairwise distinct prime numbers $p_1, \dots, p_r$ such that the element~$s_{p_i}$ has order $p_i$ for~$i = 1, \dots, r$. 
By Lemmas~\ref{lemma:splittingcayley} and~\ref{lemma:compositeorder}, we obtain the Cartesian product decomposition $G = \Cay(\langle s_{p_1} \rangle, \{s_{p_1}\}) \bx \dots \bx \Cay(\langle s_{p_r} \rangle, \{s_{p_r}\})$. 
In this decomposition, the factors are relatively prime because their orders are pairwise coprime.
Theorem~\ref{thm:cart-prod} then yields $r = 1$, and thus $G = C_{p_1}$ is a cycle.

From now on, assume $\kappa(G) > 1$. 
By \cref{lemma:cayleythreelayers}, we obtain that (at least) one of the following holds: (1) $\Gamma=\Z_3^t$ for $t\geq4$; or (2) $G$ is complete; or (3) a cycle; or (4) there exists~$s \in S$ of order at least $5$ and $|\Gamma : \langle s \rangle |\geq 5$.
In the second and third cases, there is nothing to do.
In case (1), we can choose~$s_1,s_2 \in S$ such that~$\Delta:=\langle S' \rangle := \langle s_1,s_2 \rangle$ has order~$9$ and $|\Gamma:\Delta|\geq 9$.
Similarly, in case (4), we can set~$S':=\{s\}$, so that $\Delta=\langle S' \rangle$ has order at least 5 and index at least 5. 
Consider the $K$-$l$-layer structure given by $S'$ as described in \cref{obs:group-induced-layer}. 
By Lemma~\ref{lemma:layercayley},~$K = \Cay(\Delta, S')$ is a cycle or a complete graph. The first case is excluded by Lemma~\ref{lemma:cayleycycle}. In the second, $G$ is a complete graph by Lemma~\ref{lemma:cayleycomplete}. 
\end{proof}

\subsection{Cayley graphs of even order}
\label{subsec:cayleyeven}

In the previous subsection, we studied Cayley graphs of odd order.
For the even case, we require somewhat different techniques and present them in this subsection.
Our main result for Cayley graphs of even order is the following, where we rely on a mild non-redundancy condition on the generating set.
However, we conjecture that the statement can be generalized to all Cayley graphs of abelian groups (with the additional exceptions of $K_n$ and $K_{n/2,n/2}$).

\thmcayleyeven*

We first prove Theorem~\ref{thm:cayley_even} for the case $l:=|\langle S \rangle /\langle S \setminus \{\pm s\}\rangle|\geq 4$.
Subsequently, the cases~$l\in \{2,3\}$ will be studied separately.

\begin{restatable}{lemma}{lemmacayleyevengeneral}\label{thm:main_cayley}
	Let $\Gamma$ be an abelian group of even order with generating set $S$ and~${G = \Cay(\Gamma,S)}$. Suppose there exists an $s \in S$ such that \(l = |\Gamma:\langle S \setminus \{\pm s\} \rangle| \geq 4.\) Then $G \in \mH$ if and only if~$G = C_n$ or $G = C_l \bx K_2$, where $l$ is odd or $l=4$.
\end{restatable}

\begin{proof}
	The graphs $C_n$ and $C_l \bx K_2$ for $l$ odd are in $\mH$ by Theorems~\ref{thm:cartesian-product-K2} and~\ref{thm:cyclic_generator}. 
	Now let $G \in \mH$. 
	Let $\Delta = \langle S \setminus \{ \pm s\} \rangle$ and set $k = |\Delta|$, so $|\Gamma| = |G|= kl$. 
	If $\Gamma = \langle \pm s \rangle  = \langle s \rangle$, Theorem~\ref{thm:cyclic_generator} implies $G \cong C_n$, since $G \cong K_n$ and $G \cong K_{m,m}$ (for~$m\geq 3$) both contradict the assumption that~$l\geq4$.
	Assume now $\Gamma \neq \langle s \rangle$. Then,~$k\geq 2$.
	For $k = 2$ we have $\Delta = \{0, x\} \cong \Z_2$ with $2x = 0$. 
	Thus~${S = \{ \pm s, x\}}$.
	Note that since $x \notin \langle s \rangle$ by assumption, we have $\Delta \cap \langle s \rangle  = \{0\}$ and hence $\Gamma = \Delta \times \langle s \rangle \times$. 
	Thus~$G \cong \Cay(\Delta, \{x\}) \bx \Cay(\langle s\rangle, \{\pm s\}) \cong K_2 \bx C_l$ by Lemma~\ref{lemma:splittingcayley}.
	By Theorem~\ref{thm:cartesian-product-K2},~$l$ is odd or $l=4$.
	For~$k = 3$, we argue analogously to obtain $G \cong K_3 \bx C_l$. 
	Then Theorem~\ref{thm:cart-prod} yields $l = 3$, which contradicts our assumption.
	
	Now let \(k \geq 4\) and~${K\coloneqq \Cay(\Delta,S \setminus \{\pm s\})}$.
	We show that \(G \in \mH\) leads to a contradiction. 
	We use the group-induced $K$-$l$-layer structure of $G$ defined in the beginning of \cref{sec:group-induced-layer-structure}. Since \( \langle S \setminus \{\pm s\} \rangle \subsetneq \Gamma\),
	only edges corresponding to $s$ and $-s$ connect different layers. Let~\(C\) and \(\hat{C}\) be the cycles given in \Cref{fig:construction-nonredundant}. 
	For $v,w \in V(G)$, we define $d_C(v,w)$ to be the distance of $v$ and $w$ on $C$ (i.e., the minimum number of edges on $C$ between $v$ and $w$). 
	Similarly, we define $d_{\hat{C}}(v,w)$. 
	Observe that for all $v,w \in V_2$, we have $d_{\hat{C}}(v,w) \neq k-2$, and the same holds for $V_{l-1}$, independently of the parity of $l$. 
	In the following, we color $v \in V(G)$ green (blue) if there exists $w \in N_G(v)$ with $d_C(v,w) = 2k-3$ (with $d_{\hat{C}}(v,w) = 2k-3$). 
	Only vertices in $V_2 \cup V_{l-1}$ can be colored blue.
	Clearly, an automorphism of $G$ that maps $C$ to $\hat{C}$ maps green vertices to blue vertices. 
	
	Let $a = v_{2, k-1}$, $b = v_{2,1}$, $c = v_{3,1}$ and $d = v_{3,k-1}$.
	Then $a$, $b$, $c$ and $d$ are green vertices on $C$. 
	Note that $d_C(a,b) = k-2$, $d_C(b,c) = 1$ and $d_C(c,d) = k-2$ as~$|G| = kl \geq 4k$. 
	Moreover, $a$ and $d$ are adjacent in $G$. 
	Suppose that $\varphi \in \Aut(G)$ maps~$C$ to~$\hat{C}$. 
	Let $P$ denote the path $(a,\dots, b,c, \dots, d)$ of length $2k-3$ on $C$. 
	If~$\varphi(b) \in V_2$, then~$\varphi(a) \in V_{l-1}$ as $d_{\hat{C}}(\varphi(a), \varphi(b)) = d_C(a,b) = k-2$ and no vertices in~$V_2$ have distance~$k-2$ on $\hat{C}$.	
	First assume that $\varphi(c) \in V_{l-1}$. 
	Since $d_{\hat{C}}(\varphi(b), \varphi(c)) = 1$, this implies $l = 4$ and $\varphi(b) \in \{v_{2,1}, v_{2,k}\}$. 
	But since $\varphi(a) \in V_{l-1}$ and $\varphi(a)$ and $\varphi(c)$ are on opposite sides of $\varphi(b)$ on $\varphi(P)$, we obtain~$d_{\hat{C}}(\varphi(a), \varphi(b)) > k-2$, which is a contradiction. 
	Hence $\varphi(c) \in V_2$. 	
	Then $\varphi(d) \in V_{l-1}$ as it is adjacent to $\varphi(a) \in V_{l-1}$. 
	Thus all vertices of $V(C) \setminus V(P)$ are mapped to $V_{l-1} \cup V_l$. 
	Since $\varphi(a), \varphi(d) \in V_{l-1} \cup V_l$, we have $|V(C) \setminus V(P)| \leq |V_{l-1} \cup V_l| -2 = 2k-2$.
	Then \[kl = |C| \leq |P| + |V(C) \setminus V(P) | \leq (2k-3) + (2k-2)  = 4k-5.\] 
	This implies $l \leq 3$, which is a contradiction. 
	Hence we have $\varphi(b) \in V_{l-1}$. 
	As before, we obtain~$\varphi(a) \in V_2$, $\varphi(c) \in V_{l-1}$ and $\varphi(d) \in V_2$. Then $\varphi$ maps $V(C) \setminus V(P)$ to $V_1 \cup V_2$, which leads to a contradiction as before. 
\end{proof}

\begin{figure}
	\centering
	\begin{tikzpicture}[scale=0.7]
		
		\foreach \x in {0,1,2,3,4,5} {
			\foreach \y in {0,1,2,3} {
				\draw[gray] (\x,\y) to (\x,\y+1);
			}
		}
		
		\foreach \x in {0,1,2,3,4} {
			\foreach \y in {0,1,2,3,4} {
				\draw[gray] (\x,\y) to (\x+1,\y);
			}
		}

		\foreach \y in {0,1,2,3,4} {
			\draw[gray, dashed,  bend left=40, looseness=0.5] (0,\y) to (4.9,\y);
		}
		
		\node at (-0.5,0) {\(V_{l}\)};
		
		\node at (-0.5,3) {\(V_2\)};
		\node at (-0.5,4) {\(V_1\)};
		
		\draw[decorate,decoration={brace,amplitude=10pt}] 
		(0,4.5) -- (5,4.5) node[midway,yshift=0.7cm] {\(k\)};
		
		\foreach \x in {0,1,2,3} {
			\foreach \y in {0,1,2,3,4} {
				\draw[blue, very thick] (\x,\y) to (\x+1,\y);
			}
		}
		
		\draw[blue, very thick,  bend left=40, looseness=0.5] (0,4) to (4.9,4);
		\draw[blue, very thick] (0,0) to (0,1);
		\draw[blue, very thick] (4,1) to (4,2);
		\draw[blue, very thick] (0,2) to (0,3);
		\draw[blue, very thick] (4,3) to (4,4);
		
		\draw[blue, very thick] (4,0) to (5,0);
		\foreach \y in {0,1,2,3} {
			\draw[blue, very thick] (5,\y) to (5,\y+1);
		}
		
		\foreach \x in {0,1,2,3,4,5} {
			\foreach \y in {0,1,2,3,4} {
				\node [style=vertex] (\x,\y) at (\x,\y) {};
			}
		}
		
		\node [circle,draw, thick,fill=blue!30!white, inner sep=0.75pt] (d) at (4,2) {\footnotesize \(d\)};
		\node [circle,draw, thick,fill=blue!30!white, inner sep=1.25pt] (a) at (4,3) {\footnotesize \(a\)};
		\node [circle,draw, thick,fill=blue!30!white, inner sep=0.75pt] (b) at (0,3) {\footnotesize \(b\)};
		\node [circle,draw, thick,fill=blue!30!white, inner sep=1.35pt] (c) at (0,2) {\footnotesize \(c\)};
	\end{tikzpicture}
	\hspace{2mm}
	\begin{tikzpicture}[scale=0.7]	
		\foreach \x in {0,1,2,3,4,5} {
			\foreach \y in {0,1,2,3} {
				\draw[gray] (\x,\y) to (\x,\y+1);
			}
		}
		
		\foreach \x in {0,1,2,3,4} {
			\foreach \y in {0,1,2,3,4} {
				\draw[gray] (\x,\y) to (\x+1,\y);
			}
		}
		
		\foreach \y in {0,1,2,3,4} {
			\draw[gray, dashed,  bend left=40, looseness=0.5] (0,\y) to (5,\y);
		}
		
		\node at (-0.5,0) {\(V_{l}\)};
		
		\node at (-0.5,3) {\(V_2\)};
		\node at (-0.5,4) {\(V_1\)};
		
		\draw[decorate,decoration={brace,amplitude=10pt}] 
		(0,4.5) -- (5,4.5) node[midway,yshift=0.7cm] {\(k\)};
		
		\foreach \x in {0,1,3,4} {
			\foreach \y in {0,1,2,3,4} {
				\draw[blue, very thick] (\x,\y) to (\x+1,\y);
			}
		}
		
		\draw[blue, very thick,  bend left=40, looseness=0.5] (0,4) to (5,4);
		\draw[blue, very thick] (0,0) to (0,1);
		\draw[blue, very thick] (5,0) to (5,1);
		\draw[blue, very thick] (2,1) to (2,2);
		\draw[blue, very thick] (3,1) to (3,2);
		\draw[blue, very thick] (0,2) to (0,3);
		\draw[blue, very thick] (5,2) to (5,3);
		\draw[blue, very thick] (2,3) to (2,4);
		\draw[blue, very thick] (3,3) to (3,4);
		
		\draw[blue, very thick] (2,0) to (3,0);
		
		\foreach \x in {0,1,2,3,4,5} {
			\foreach \y in {0,1,2,3,4} {
				\node [style=vertex] (\x,\y) at (\x,\y) {};
			}
		}
	\end{tikzpicture}
	\hspace{2mm}
	\begin{tikzpicture}[scale=0.7]	
		\useasboundingbox (-0.8,-0.85) rectangle (4.2,5);
		\foreach \x in {0,1,2,3,4} {
			\foreach \y in {0,1,2} {
				\draw[gray] (\x,\y) to (\x,\y+1);
			}
		}
		
		\foreach \x in {0,1,2,3} {
			\foreach \y in {0,1,2,3} {
				\draw[gray] (\x,\y) to (\x+1,\y);
			}
		}
		
		\foreach \y in {0,1,2} {
			\draw[gray, dashed,  bend left=40, looseness=0.5] (0,\y) to (4,\y);
		}
		
		\node at (-0.5,0) {\(V_{l}\)};
		\node at (-0.5,2) {\(V_2\)};
		\node at (-0.5,3) {\(V_1\)};
		
		\draw[decorate,decoration={brace,amplitude=10pt}]
		(0,3.5) -- (4,3.5) node[midway,yshift=0.7cm] {\(k\)};
		
		\foreach \x in {0,1,3} {
			\foreach \y in {0,1,2,3} {
				\draw[blue, very thick] (\x,\y) to (\x+1,\y);
			}
		}
		\draw[blue, very thick,  bend left=40, looseness=0.5] (0,3) to (4,3);
		\draw[blue, very thick,  bend left=40, looseness=0.5] (0,0) to (4,0);
		\draw[blue, very thick] (2,0) to (2,1);
		\draw[blue, very thick] (3,0) to (3,1);
		\draw[blue, very thick] (0,1) to (0,2);
		\draw[blue, very thick] (4,1) to (4,2);
		\draw[blue, very thick] (2,2) to (2,3);
		\draw[blue, very thick] (3,2) to (3,3);
		
		\foreach \x in {0,1,2,3,4} {
			\foreach \y in {0,1,2,3} {
				\node [style=vertex] (\x,\y) at (\x,\y) {};
			}
		}
	\end{tikzpicture}
	\caption{Illustrations for the Hamiltonian cycles in the proof of~\cref{thm:main_cayley}. The left subfigure depicts the cycle $C$, the middle one $\hat{C}$ in case $k$ even, and the right one $\hat{C}$ for $k$ odd. The vertical edges in the middle of \(\hat{C}\) are at columns \(\lceil k/2 \rceil\) and \(\lceil k/2 \rceil +1\).}
	\label{fig:construction-nonredundant}
\end{figure} 

In the next two lemmas, we settle the cases where $l\in\{2,3\}$.

\begin{restatable}{lemma}{lemmacayleyevenltwo}\label{lemma:l=2}
	Let $\Gamma$ be an abelian group with generating set $S$ and $G = \Cay(\Gamma,S)$. Suppose that there exists an $s \in S$ such that \(l = |\Gamma:\langle S \setminus \{\pm s\} \rangle| =2.\) Then $G \in \mH$ if and only if~$G \in \{C_4, K_{4,4}\}$ or $G = C_k \bx K_2$, where $k$ is odd or $k=4$.
\end{restatable}

\begin{proof}
	First, note that $C_4, K_{4,4}\in \mH$ and by Lemma~\ref{lem:oddcycle_prism}, we have $C_k \bx K_2 \in \mH$ when~$k$ is odd.
	It is left to prove the other direction of the assertion.
	For this, assume that $G \in \mH$.
	Let~$S' = S \setminus \{\pm s\}$, let $\Delta = \langle S' \rangle$, and set $K = \Cay(\Delta, S')$. Then $G$ has a group-induced $K$-$2$-layer structure as defined in the beginning of \cref{sec:group-induced-layer-structure}. Moreover, every vertex~$v$ has at most two neighbors on the other layer given by $v+s$ and $v-s$.
	In case~$s=-s$, we have~${G=K\bx K_2}$. So, either $G = K_2 \notin \mH$, $G=C_4$ or the assertion follows by \cref{thm:cartesian-product-K2}. Therefore, assume $s\neq -s$.
	In other words, $G$ contains $K\bx K_2$ with an additional disjoint matching between the two layers such that every vertex has precisely two neighbors in the other layer.
	Note that this immediately implies that $G$ cannot be the complete graph or a cycle.
	
	In this proof, we use some of the cycles that we already constructed for the proofs in \cref{sec:cartesian-K2}, and therefore, we use a similar notation.
	More precisely,~$G$ contains two copies of~$K$ and we denote the vertices in $G$ by $v_0, \dots, v_{k-1}, u_0, \dots, u_{k-1}$ where $v_0, \dots, v_{k-1}$ are the vertices in one copy of $K$ and $u_0, \dots, u_{k-1}$ are vertices in the other copy of $K$.
	Hereby, we number the vertices such that the isomorphism between the two copies is given by $u_i+s=v_i$.
	When we only talk about the graph $K$ (not as a subgraph of $G$), we denote the vertices by $w_0, \dots, w_{k-1}$ and hereby, we again use numbering such that there is an isomorphism mapping $w_i$ to $v_i$.
	We fix a Hamiltonian cycle $C_K$ of $K$ and number the vertices such that~$C_K=(w_0, \dots, w_{k-1})$.
	
	\textbf{Odd case.} First assume that $k$ is odd. 
	Using the cycle $C_K$ of $K$, we consider the Hamiltonian cycle of $G$ depicted in \cref{fig:prism}, which we denote by $C$.
	As $|G|=2k$ is even,~$\kappa(G)$ is even by \cref{lemma:divisibilitylemma}, and hence rotating~$C$ by~$k$ is an automorphism of~$G$. 
	This automorphism maps the first layer to the second layer and, therefore, induces an automorphism of $K$.
	This induced automorphism maps $w_i$ to~$w_{k-1-i}$, i.e., it is a reflection of $C_K$ at vertex $w_{(k-1)/2}$.
	By shifting the numbering of the vertices in~$C_K$, we obtain that~$\Aut(K)$ contains the reflections at every vertex of~$C_K$. 
	As $k$ is odd, these generate the dihedral group of order $2k$. This means that rotating $C_K$ by one is an automorphism of~$K$. By \cite[Section~1.4]{GRE24}, $K$ is the Cayley graph of a cyclic group and $S'$ contains a generator. 
	By Theorem~\ref{thm:cyclic_generator}, we have $K \cong C_k$ or~$K \cong K_k$.
	
	First assume that $K \cong K_k$.
	If $k>3$, we obtain that $G$ is a complete graph by Lemma~\ref{lemma:cayleycomplete} and we have already argued above that this leads to a contradiction.
	If~$k\leq 3$, we even have $k=3$ so that $|G|=6$. By \cref{lemma:small-groups}, there are no Cayley graphs of abelian groups in~$\mH$ on 6 vertices, except for $K_6, C_6, K_{3,3}$ and $K_3 \bx K_2$.
	As already argued, both $G=K_6$ and $G=C_6$ lead to a contradiction. Note that~$G=K_{3,3}$ leads to a contradiction as well: 
	we already restricted to the case where every vertex has precisely two neighbors in the other layer, and each vertex also has at least two neighbors in the same layer. Since every vertex in $K_{3,3}$ has degree 3, this is not possible.
	Therefore, the only case that remains here is $G=K_3 \bx K_2$.
	
	Therefore, assume that $K \cong C_k$.
	Let $v_j:=u_0-s$, i.e., the two neighbors of $u_0$ in the other layer are $v_0$ and $v_j$.
	Since each layer is a Cayley graph, which is a cycle, we obtain that~$S$ contains only one other element than $s$, which we denote by $s'$, and we have $K=\langle s' \rangle$.
	We obtain that $v_j=u_0-s=u_0+s+js'$ so that $-s=s+js'$.
	This proves that, for every~$u_i$, its two neighbors in the other layer are $v_i$ and $v_{i+j\bmod k}$.
	If $j\in\{1,2\}$, consider the Hamiltonian cycle of $G$ given by
	\begin{equation*}
		C^*=(v_0, u_0, v_j, u_j, v_{2j}, u_{2j}, \dots, v_{-j \bmod k}, u_{-j \bmod k}).
	\end{equation*}
	Note that $C^*$ has the following property: Every two vertices $a,a'$ with $d_{C^*}(a,a')=2$ are connected by a path of length $j$ that is not contained in the cycle.
	For $j\in\{1,2\}$ it is easy to see that the cycle $C$ does not have this property because it fails, e.g., for the vertices~$a=v_0$ and $a'=v_2$ (see \cref{fig:prism}).
	
	\begin{figure}
		\centering
		\begin{tikzpicture}[scale=1, every node/.style={circle, draw, inner sep=0pt, minimum size=3.5mm}]
			\foreach \x in {0,1,2,3,4,5,6,7,8} {
				\node (v\x) at (\x,1) {\footnotesize{$v_{\x}$}};
				\node (u\x) at (\x,0) {\footnotesize{$u_{\x}$}};
				\draw[gray] (v\x) to (u\x);
			}
			\foreach \x in {0,1,2,3,4,5,6,7} {
				\draw[gray] (v\x) to (v\the\numexpr \x+1 \relax);
				\draw[gray] (u\x) to (u\the\numexpr \x+1 \relax);
			}
			\draw[blue, very thick] (v0) to (v1) to (u1) to (u2) to (v2) to (v3) to (u3) to (u4) to (v4) to (u0) to [bend right=15] (u8);
			\draw[blue, very thick] (u8) to (u7) to (u6) to (u5) to (v5) to (v6) to (v7) to (v8) to [bend right=15] (v0);
		\end{tikzpicture}
		\caption{Hamiltonian cycle in the proof of \cref{lemma:l=2} for odd $k$, depicted for $j=4$.}\label{fig:l=2_k-odd}
	\end{figure}
	
	Assume now that $j\geq 3$. Since $k$ is odd, by renumbering of the vertices, we can assume that $j$ is even and $4\leq j\leq k-3$.
	Consider the cycle depicted in \cref{fig:l=2_k-odd}.
	Note that, in this cycle, the pairs of vertices $\{v_1,v_2\}$, $\{u_2, u_3\}$, $\{v_3,v_4\}$ have distance~$3$ along the cycle and are neighboring.
	Opposed to this, in~$C$, there are only two pairs of vertices with this property, given by $\{v_1, u_1\}, \{v_{k-2}, u_{k-2}\}$.
	(Note that we have used $4\leq j\leq k-3$ to exclude that~$v_0$ and $u_2$ are adjacent.)
	This completes the proof of the assertion in the case where $k$ is odd.
	
	\textbf{Even case.}~From now on, we assume that~$k = 2k'$ is even. 
	For $k\leq 4$, one can check that the only additional Hamiltonian-transitive graph (given in \cref{lemma:small-groups}) arising in this case is $K_{4,4}$.
	Therefore, assume that $k\geq 6$.
	Again, fix a cycle~$C$ as in the middle of \cref{fig:cartesian-K2-even} and consider the cycle~$C'$ illustrated on the left side of \cref{fig:cartesian-K2-even}. 
	Every second edge on $C'$ lies completely in one of the layers. Those edges have the property that its neighbors on~$C'$ to both sides are adjacent. Hence, also every second edge of~$C$ fulfills this property. In the following, edges on $C$ with this property are called \emph{marked}.
	
	We show that $v_0$ and $u_2$ are adjacent.
	First assume that, among others, the edges $\{v_1, v_2\}, \{v_3, v_4\}, \dots, \{u_1, u_2\}, \{u_3, u_4\}, \dots$ are marked. Then~\(\{v_{k-3}, v_{k-2}\}\) is a marked edge and we have \(\{v_{k-4},v_{k-1}\} \in E(G)\).
	Consider the cycle~$C''$ depicted in Figure~\ref{fig:graphenl=2} that uses the edge $\{v_{k-3}, v_{k-2}\}$. 
	Rotating~$C''$ by $k$ maps the edge~$\{v_{k-2}, v_{k-1}\}$ to $\{v_0, u_2\}$, so~$v_0$ and $u_2$ are adjacent. 
	In case that the edges $\{v_0, v_1\}, \{v_2, v_3\}, \dots,$ $\{u_0, u_1\}, \{u_2, u_3\}, \dots$ are marked, it follows directly that $v_0$ and $u_2$ are adjacent by~$\{u_0,u_1\}$ being marked. 
	By symmetry in $u$ and $v$, we obtain also that $u_0$ and $v_2$ are adjacent. 
	Then, by the same argument with cycle
	~$C=(v_{k-2},v_{k-1},v_0,v_1,\dots,v_{k-3},$ $u_{k-3},u_{k-4},\dots,u_0,u_{k-1},u_{k-2})$ we obtain $u_{k-2}$ is adjacent to $v_0$. 
	Thus, $v_0$ has three neighbors on the second layer, which is a contradiction. 
\end{proof}

\begin{figure}
	\centering
	\begin{tikzpicture}[scale=1.15, every node/.style={circle, draw, inner sep=0pt, minimum size=3.5mm}]
		\foreach \x in {0,1,2,3,4,5,6,7} {
			\node (v\x) at (\x,1) {\footnotesize{$v_{\x}$}};
			\node (u\x) at (\x,0) {\footnotesize{$u_{\x}$}};
		}
		\foreach \x in {0,1,2,3,4,5,6,7} {
			\draw[gray] (v\x) to (u\x);
		}
		\draw[gray] (v4) to (v5);
		\draw[gray] (v6) to (v7);
		\draw[gray] (u6) to (u5);
		\draw[blue, very thick] (v0) to (u0) to (u1) to (u2) to (u3) to (u4) to (u5) to (v5) to (v6) to (u6) to (u7) to (v7) to [bend right=20] (v4);
		\draw[blue, very thick] (v4) to (v3) to (v2) to (v1) to (v0);
	\end{tikzpicture}
	\caption{Hamiltonian cycle in the proof of Lemma~\ref{lemma:l=2} for even $k$, depicted for $k = 8$.}
	\label{fig:graphenl=2}
\end{figure}

\begin{restatable}{lemma}{lemmacayleyevenlthree}\label{lemma:l=3}
	Let $\Gamma$ be an abelian group of even order with generating set $S$ and~${G= \Cay(\Gamma,S)}$. Suppose there exists an $s \in S$ such that \(l = |\Gamma:\langle S \setminus \{\pm s\} \rangle| =3.\) Then $G \in \mH$ if and only if~$G \in \{K_{3,3}, C_3 \bx K_2\}$.
\end{restatable}

\begin{proof}
	As in the proof of Lemma~\ref{lemma:l=2}, let $S' = S \setminus \langle \pm s \rangle$, $K:=\Cay(\langle S' \rangle, S')$ and~${k = |K|}$. Then~$G$ has a $K$-$l$-layer structure.
	Since $|\Gamma|=lk=3k$ is even, we obtain that $k = 2k'$ for some~\(k' \in \N\). 
	The case $k' = 1$ is covered by Lemma~\ref{lemma:small-groups} yielding $G=K_{3,3}$ or $G=C_3 \bx K_2$ \linebreak since $G$ being complete or a cycle contradicts $l=3$. Thus, assume $k' \geq 2$. 
	Consider the Hamiltonian cycle $C$ depicted in Figure~\ref{fig:graphenl=3_kreis}. 
	Since $|G|$ is even, also $\kappa(G)$ is even by Lemma~\ref{lemma:divisibilitylemma}.
	Thus, $C$ is $2$-symmetric. 
	The corresponding rotation by $3k'$	maps the edge~$\{v_{k'}, u_{k'}\}$ to~$\{w_{k-1},u_0\}$. 
	Since the only neighbor of $u_0$ on the third layer is $w_0$, this is a contradiction.
\end{proof}

\cref{thm:cayley_even} now follows immediately by combining Lemmas~\ref{thm:main_cayley}, \ref{lemma:l=2}, and \ref{lemma:l=3}.

\section{Constructions of regular Hamiltonian-transitive graphs}\label{sec:constructions}

In this section, we give two constructions for families of regular Hamiltonian-transitive graphs. Somewhat in contrast to Sections~\ref{sec:cartprod} and~\ref{sec:cayley}, this underlines that there are many regular Hamiltonian-transitive graphs.

\begin{remark}
	\label{rem:reg-graphs-in-H}
	For every \(d \geq 2\), there are infinitely many \(d\)-regular Hamiltonian-transitive graphs in \(\mH\): We construct such graphs as follows. For an arbitrary $n\geq 3$, take \(C_n\) and \(n\) disjoint copies of \(K_{d+1}\). 
	Remove one edge from each copy of \(K_{d+1}\) and replace each vertex~$v$ of \(C_n\) by a copy of \(K_{d+1}\) such that the two incident edges to $v$ are now incident to the two vertices of degree \(d-1\). See~\Cref{fig:constr-d-reg} for an illustration.
	This yields a \(d\)-regular Hamiltonian graph. It is Hamiltonian-transitive because every Hamiltonian cycle uses all of the edges of~$C_n$ and visits the vertices in each copy of $K_{d+1}$ in some arbitrary order. It is immediate that this order can be permuted by applying an automorphism.
\end{remark}

\begin{figure}[h!]
	\begin{center}
		\begin{tikzpicture}[scale=0.7]	
			\foreach \i in {1,...,5} {
				\node[draw=black, fill=white, shape=circle, inner sep=2.5pt] (\i) at ({360/20 * (4*\i+1)}:2) {\scriptsize \(K_{d+1}\)};
			}
			
			\foreach \i in {1,...,10} {
				\node[style=vertex, fill=white] (\i) at ({360/10 * (\i-1)}:2) {};
			}
			
			\foreach \i/\j in {2/3,4/5,6/7,8/9,10/1}{
				\draw[thick] (\i) to (\j);
			}
			
			\foreach \i/\j in {1/2,3/4,5/6,7/8,9/10}{
				\draw[line width=1.5pt, red, opacity=0.5] (\i) to (\j);
			}
			
			\foreach \i in {1,...,5} {
				\node[draw=black, fill=none, shape=circle, inner sep=2.5pt] (\i) at ({360/20 * (4*\i+1)}:2) {\scriptsize \(K_{d+1}\)};
			}
			
			\foreach \i in {1,...,10} {
				\node[style=vertex, fill=white] (\i) at ({360/10 * (\i-1)}:2) {};
			}
		\end{tikzpicture}
		\hspace{0.5cm}
		\begin{tikzpicture}[scale=0.7]
			\foreach \i in {1,...,10} {
				\node[style=vertex] (\i) at ({360/10 * (\i-1)}:2) {};
			}
			
			\foreach \j in {1,...,20} {
				\ifthenelse{\j=2 \OR \j=6 \OR \j=10 \OR \j=14 \OR \j=18}{
					\node[style=vertex, fill=white] (\j;in) at ({360/20 * (\j - 1)}:1.5) {};
					\node[style=vertex, fill=white] (\j;out) at ({360/20 * (\j - 1)}:2.5) {};
					\draw[thick] (\j;in) to (\j;out);
				}
				{}
			}
			
			\foreach \i/\j in {2/3,4/5,6/7,8/9,10/1}{
				\draw[thick] (\i) to (\j);
			}
			
			\foreach \i/\j in {1/2,3/6,5/10,7/14,9/18}{
				\draw[thick] (\i) to (\j;out);
				\draw[thick] (\i) to (\j;in);
			}
			
			\foreach \i/\j in {2/2,4/6,6/10,8/14,10/18}{
				\draw[thick] (\i) to (\j;out);
				\draw[thick] (\i) to (\j;in);
			}
		\end{tikzpicture}
	\end{center}
	\caption{The construction of the graphs in~\cref{rem:reg-graphs-in-H} for \(n=5\). In each \(K_{d+1}\) the removed edge is indicated in red. On the right you can see the construction for \(d=3.\)}
	\label{fig:constr-d-reg}
\end{figure}
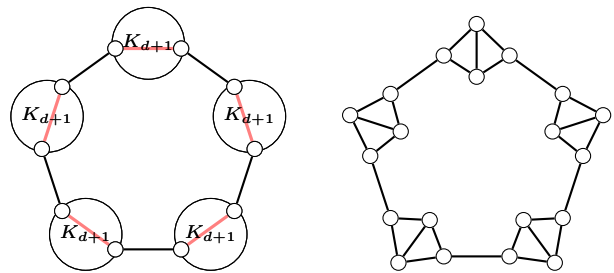

This is in strong contrast to uniquely Hamiltonian graphs, where it is known that no uniquely Hamiltonian \(d\)-regular graphs exist for all odd values of \(d\) and all \(d > 22\) \cite{Thomason78,Tutte46,SheConj22}, and Sheehan's Conjecture asserts that no such graphs exist for any \(d \geq 3\)~\cite{SHE75}.

Note that the constructed graphs in \cref{rem:reg-graphs-in-H} have 2-edge separators and this structure forces certain edges to be contained in every Hamiltonian cycle.
Opposed to this, for cubic graphs, we construct an infinite family of 3-connected Hamiltonian-transitive graphs.
This shows that there are many different cubic Hamiltonian-transitive graphs besides the odd prisms from \cref{thm:cayley_even}. In particular, this construction yields a family of Hamiltonian-transitive graphs with an unbounded number of edge orbits. 
An \textit{edge orbit} of a graph is an equivalence class of edges under the action of~\(\Aut(G)\). More precisely, we show the following. 

\begin{restatable}{proposition}{cubgraphsinH}
	\label{thm:cub-graphs-in-H}	
	There exists a family of cubic 3-connected Hamiltonian-tran\-si\-tive graphs with an unbounded number of edge orbits.
\end{restatable}

For this, we use the truncation of a cubic graph, which is a classical concept often used in the literature (e.g., \cite{EibenJS19,GuoMohar2024,Sachs63}). The \emph{truncation} \(\trunc(G)\) of a cubic graph \(G\) is the cubic graph obtained by replacing each vertex~\(v\) in \(G\) with a triangle and adding for each edge~\(\{u,v\}\) of \(G\) an edge between two vertices of the triangles corresponding to \(u\) and \(v\). This is illustrated in \cref{fig:truncation}. For a vertex \(v\) in \(G\) with neighbors~\(\{x,y,z\}\), we denote the corresponding vertices in~\(\trunc(G)\) by~\(\{v_x,v_y,v_z\}\).
Note that the truncation of a~3-connected graph is again~3-connected.

\begin{figure}
	\centering
	\begin{tikzpicture}[scale=0.9]
		\useasboundingbox (-0.1,-0.35) rectangle (3.1,3.35);
		\draw[thick] (0.75,0.75) -- (0.75,0) -- (0,0.75) -- (0.75,0.75);
		\draw[thick] (0.75,2.25) -- (0,2.25) -- (0.75,3) -- (0.75,2.25);
		\draw[thick] (2.25,2.25) -- (2.25,3) -- (3,2.25) -- (2.25,2.25);
		\draw[thick] (2.25,0.75) -- (3,0.75) -- (2.25,0) -- (2.25,0.75);
		\draw[thick] (0.75,0.75) -- (2.25,2.25);
		\draw[thick] (0.75,2.25) -- (2.25,0.75);
		\draw[thick] (0.75,0) -- (2.25,0);
		\draw[thick] (0.75,3) -- (2.25,3);
		\draw[thick] (3,2.25) -- (3,0.75);
		\draw[thick] (0,0.75) -- (0,2.25);
		\foreach \x in {0.75,2.25} {
			\foreach \y in {0,0.75,2.25,3} {
				\node [style=vertex] (\x,\y) at (\x,\y) {};
			}
		}
		\foreach \x in {0,3} {
			\foreach \y in {0.75,2.25} {
				\node [style=vertex] (\x,\y) at (\x,\y) {};
			}
		}
	\end{tikzpicture}
	\hspace{5mm}
	\begin{tikzpicture}[scale=0.9]
		\useasboundingbox (-0.1,-0.35) rectangle (5.1,3.35);
		\draw[thick] (0.75,2.25) -- (0,2.25) -- (0.75,3) -- (0.75,2.25);
		\draw[thick] (2.5,2.25) -- (2,3) -- (3,3) -- (2.5,2.25);
		\draw[thick] (4.25,2.25) -- (4.25,3) -- (5,2.25) -- (4.25,2.25);
		
		\draw[thick] (0.75,0.75) -- (0,0.75) -- (0.75,0) -- (0.75,0.75);
		\draw[thick] (2.5,0.75) -- (2,0) -- (3,0) -- (2.5,0.75);
		\draw[thick] (4.25,0.75) -- (4.25,0) -- (5,0.75) -- (4.25,0.75);
		\foreach \x in {0.75,3} {
			\foreach \y in {0,3} {
				\draw[thick] (\x,\y) -- (\x+1.25,\y);
			}
		}
		\foreach \x in {0,2.5,5} {
			\draw[thick] (\x,2.25) -- (\x,0.75);
		}
		\draw[thick] (0.75,0.75) -- (4.25,2.25);
		\draw[thick] (0.75,2.25) -- (4.25,0.75);
		
		\foreach \x in {0.75,2,3,4.25} {
			\foreach \y in {0,3} {
				\node [style=vertex] (\x,\y) at (\x,\y) {};
			}
		}
		\foreach \x in {0,0.75,2.5,4.25,5} {
			\foreach \y in {0.75,2.25} {
				\node [style=vertex] (\x,\y) at (\x,\y) {};
			}
		}
	\end{tikzpicture}
	\hspace{5mm}
	\raisebox{1ex}{
		\begin{tikzpicture}[scale=0.8]
			\draw[thick] (0.5,0.5) -- (0.5,0) -- (0,0.5) -- (0.5,0.5);
			\draw[thick] (0.5,3) -- (0,3) -- (0.5,3.5) -- (0.5,3);
			\draw[thick] (3,3) -- (3,3.5) -- (3.5,3) -- (3,3);
			\draw[thick] (3,0.5) -- (3.5,0.5) -- (3,0) -- (3,0.5);
			
			\draw[thick] (1,1) -- (1.5,1) -- (1,1.5) -- (1,1);
			\draw[thick] (1,2.5) -- (1,2) -- (1.5,2.5) -- (1,2.5);
			\draw[thick] (2.5,2.5) -- (2,2.5) -- (2.5,2) -- (2.5,2.5);
			\draw[thick] (2.5,1) -- (2,1) -- (2.5,1.5) -- (2.5,1);
			
			\draw[thick] (0.5,0.5) -- (1,1);
			\draw[thick] (2.5,2.5) -- (3,3);
			\draw[thick] (0.5,3) -- (1,2.5);
			\draw[thick] (2.5,1) -- (3,0.5);
			
			\draw[thick] (0.5,0) -- (3,0);
			\draw[thick] (0.5,3.5) -- (3,3.5);
			\draw[thick] (3.5,3) -- (3.5,0.5);
			\draw[thick] (0,0.5) -- (0,3);
			
			\draw[thick] (1.5,1) -- (2,1);
			\draw[thick] (1.5,2.5) -- (2,2.5);
			\draw[thick] (2.5,2) -- (2.5,1.5);
			\draw[thick] (1,1.5) -- (1,2);
			
			\foreach \x in {0,3.5} {
				\foreach \y in {0.5,3} {
					\node [style=vertex] (\x,\y) at (\x,\y) {};
				}
			}
			\foreach \x in {0.5,3} {
				\foreach \y in {0,0.5,3,3.5} {
					\node [style=vertex] (\x,\y) at (\x,\y) {};
				}
			}
			\foreach \x in {1,2.5} {
				\foreach \y in {1,1.5,2,2.5} {
					\node [style=vertex] (\x,\y) at (\x,\y) {};
				}
			}
			\foreach \x in {1.5,2} {
				\foreach \y in {1,2.5} {
					\node [style=vertex] (\x,\y) at (\x,\y) {};
				}
			}
		\end{tikzpicture}
	}
	\caption{The truncation of \(K_4\), \(K_{3,3}\), and the cube. These provide further examples of cubic vertex-transitive Hamiltonian-transitive graphs besides the ones from \cref{thm:cayley_even}.}
	\label{fig:truncation}
\end{figure}
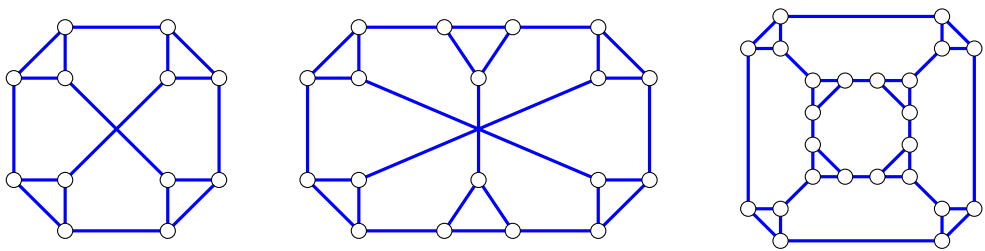	

Now we show that the truncation operation preserves Hamiltonian transitivity. Then repeatedly applying it to a Hamiltonian-transitive base graph yields a family of Hamiltonian-transitive graphs.

\begin{restatable}{lemma}{truncinH}
	\label{lem:trunc-in-H}	
	If $G$ is a cubic graph in $\mH$, then $\trunc(G)\in \mH$.
\end{restatable}

\begin{proof}
	First, observe that a Hamiltonian cycle \(C\) in $G$ naturally extends to a Hamiltonian cycle \(\tilde{C}\) in \(\trunc(G)\): For each three consecutive vertices \((u,v,w)\) in \(C\), choose the corresponding sequence of vertices \((u_v,v_u,v_x,v_w,w_v)\) in \(\trunc(G)\), where \(x\) is the third neighbor of \(v\) in \(G\) besides \(u\) and \(w\). 
	
	Indeed, every Hamiltonian cycle in \(\trunc(G)\) is such a natural extension of a Hamiltonian cycle in \(G\):
	Since every triangle \(T\) in $\trunc(G)$ has precisely three edges to vertices in \(V(\trunc(G)) \setminus \{T\}\), it is traversed exactly once by every Hamiltonian cycle. Thus every Hamiltonian cycle in $\trunc(G)$ contains the subsequences \((v_x,v_y,v_z)\) for every vertex \(v \in G\) and its three neighbors~\(\{x,y,z\}\). Hence it is an extension of a Hamiltonian cycle in $G$.
	
	Now let \(\varphi\) be an automorphism of \(G\) mapping a Hamiltonian cycle \(C_1\) to another Hamiltonian cycle \(C_2\). Let \(\tilde{\varphi} : \trunc(G) \to \trunc(G)\) with \(\tilde{\varphi}(v_u)=\varphi(v)_{\varphi(u)}\). This is an automorphism of~\(\trunc(G)\) since two vertices \(u_v\) and \(x_y\) are adjacent in \(\trunc(G)\) if and only if (i) \(u=x\) or (ii)~\(u=y\), \(v=x\) and \(\{u,v\} \in E(G)\). Further, \(\tilde{\varphi}\) maps \(\tilde{C_1}\) to \(\tilde{C_2}\) since \(\varphi\) maps \(C_1\) to \(C_2\).
	
	Together we obtain that every two Hamiltonian cycles of \(\trunc(G)\) are of the form~\(\tilde{C_1}\) and~\(\tilde{C_2}\) for Hamiltonian cycles~\(C_1\) and~\(C_2\) of \(G\). Since \(G\) is in \(\mH\), there is an automorphism~\(\varphi\) of \(G\) mapping~\(C_1\) to~\(C_2\). Then~\(\tilde{\varphi}\) is an automorphism of~\(\trunc(G)\) mapping~\(\tilde{C_1}\) to \(\tilde{C_2}\). Hence~\(\trunc(G)\) is in~\(\mH\).
\end{proof}

Next, we argue that the truncation operation increases the number of edge orbits. While this is well-known, we include the proof for self-containment.
For an edge~\(e\), we denote by \(\ell(e)\) the length of the shortest cycle containing~\(e\). This clearly is invariant under graph automorphisms.

\begin{lemma}
	\label{lem:trunc-ell}
	For every edge \(\{u,v\}\) in a cubic graph \(G\), the corresponding edge \(\{u_v,v_u\}\) in~\(\trunc(G)\) satisfies~\(\ell(\{u_v,v_u\})= 2 \ell(\{u,v\})\).
\end{lemma}

\begin{proof}
	A cycle \(C\) in \(G\) naturally corresponds to a cycle \(\hat{C}\) in \(\trunc(G)\), by choosing for each three consecutive vertices \((u,v,w)\) in \(C\), the vertices \((u_v,v_u,v_w,w_v)\) to be in \(\hat{C}\). Then \(|\hat{C}| = 2|C|\). 
	Having this, we directly obtain \(\ell(\{u_v,v_u\}) \leq 2 \ell(\{u,v\})\), by taking a cycle \(C\) containing~\(\{u,v\}\) of length~\(\ell(\{u,v\})\) and then considering the cycle~\(\hat{C}\) in~\(\trunc(G)\). This clearly contains~\(\{u_v,v_u\}\) and is of length~\(2 \ell(\{u,v\})\). 
	Further, to every cycle \(C'\) in \(\trunc(G)\) of length greater 3 there is a natural corresponding cycle \(\pi(C')\) in~\(G\) of at most half the length by shrinking the triangles to vertices and removing the second and possibly third occurrence of a vertex forgetting about the information of the truncation.
	Now assume there is a cycle \(C'\) in \(\trunc(G)\) of length less than \(2 \ell(\{u,v\})\) containing \(\{u_v,v_u\}\). Then \(\pi(C')\) is a cycle in \(G\) of length less then \(\ell(\{u,v\})\) containing \(\{u,v\}\), which is a contradiction.
\end{proof}

\begin{restatable}{lemma}{truncedgeorbit}
	\label{lem:trunc-edge-orbits}	
	Let \(k \geq 3\) and \(G\) be a cubic graph such that all edges \(e\) of \(G\) satisfy \(\ell(e) = k\). Then the family \(\{\trunc^n(G)\}_{n \in \N}\) has an unbounded number of edge orbits.
\end{restatable}

\begin{proof}
	By induction we show that \(\trunc^n(G)\) has at least \(n\) edges with pairwise different values for~\(\ell\). Since \(\ell\) is invariant under graph automorphism, this then yields that~\(\trunc^n(G)\) has at least \(n\) edge orbits.
	
	The statement for the base case \(n=1\) is trivial. Now, assume that it holds for~\(n\). Then we find \(n\) edges \(\{\{u^1,v^1\}, \dots, \{u^n,v^n\}\}\) in \(\trunc^n(G)\) with pairwise distinct values for \(\ell\).
	Now, by applying \cref{lem:trunc-ell} every edge \(\{u^i_{v^i},v^i_{u^i}\}\) in \(\trunc^n(G)\) satisfies \(\ell(e)=2\ell(\{u^i,v^i\})\). This gives \(n\) edges with pairwise distinct values of \(\ell\), each of which is larger than \(3\). Further, by taking any vertex \(x \in V(G)\) with neighbors \(y\) and \(z\), the edge \(\{x_y,x_z\} \in E(\trunc(G))\) has~\(\ell(\{x_y,x_z\})=3\). Thus we find \(n+1\) edges in \(\trunc(G)\) with pairwise distinct values for~\(\ell\).
\end{proof}

Now we have all tools at hand to prove \cref{thm:cub-graphs-in-H}.

\begin{proof}[Proof of \cref{thm:cub-graphs-in-H}]
	Choose \(K_4\) as a base graph. Then \(\{\trunc^n(K_4)\}_{n \in \N}\) is a family of Hamiltonian-transitive graphs by \cref{lem:trunc-in-H}. Further, this family has unbounded number of edge orbits by~\cref{lem:trunc-edge-orbits}.
\end{proof}

Repeatedly applying the truncation operation to other cubic Hamiltonian-tra\-nsi\-tive base graphs besides \(K_4\), e.g., to odd prisms, yields even more families of graphs in \(\mH\).


\section{Graphs with many Hamiltonian cycles up to symmetry}\label{sec:many-cycles}

In the previous sections, we studied which graphs have a unique Hamiltonian cycle up to symmetry. 
More generally speaking, we investigated the number of Hamiltonian cycles of a graph up to symmetry, with a focus on the special cases where this number is one. 
In this section, we take a look at the other end of the spectrum, i.e., we consider graphs that have as many Hamiltonian cycles up to symmetry as possible.

Let~\(G\) be a graph with~\(n\) vertices. Then,~\(G\) has at most~\(\frac{n!}{2n}\) many Hamiltonian cycles. For complete graphs, this bound is tight. Thus, asymptotically the maximum number of Hamiltonian cycles of a graph on~\(n\) vertices is in~\( \Theta((n-1)!)=2^{\Theta(n \log n)}\). We show that there are graphs that asymptotically have this many Hamiltonian cycles even up to symmetry.

\begin{restatable}{proposition}{manycycles}
	\label{prop:many-ham-cyc}	
	For every \(n \in \N\), there are graphs on~\(n\) vertices with~\(2^{\Theta(n \log n)}\) many Hamiltonian cycles up to symmetry.
\end{restatable}

\begin{proof}
	Consider the graph~\(G_n=K_n \setminus C_n\). Since the automorphism group is invariant under taking complements, we have~\(|\Aut(G_n)|=|\Aut(C_n)|=2n\).
	On the other hand, the number of Hamiltonian cycles in~\(G_n\) is lower bounded by~\(\frac{(n-3)!}{12}\): To see this, label the vertices of~\(G_n\) by~\(\{1,\dots, n\}\) and write a Hamiltonian cycle as an ordered sequence of vertices beginning with~1.
	Since every vertex has two non-neighbors, there are~\(\frac{(n-3)!}{3!}\) ordered sequences of pairwise distinct vertices of length \(n-5\) beginning with 1, such that neighboring vertices in the sequence are adjacent in the graph. Every such sequence can be extended to a Hamiltonian cycle of the graph: The remaining five vertices contain \(K_5 \setminus P_5\) as a subgraph. In particular, they lie on a cycle of length five. Further, the first and last vertex of the sequence are each connected to at least three of the five vertices. Thus, there are two neighboring vertices on the cycle one of which is connected to the first vertex of the sequence and the other one is connected to the first vertex of the sequence. Hence, the sequence can be extended to a Hamiltonian cycle.
	Since every cycle appears in two orientations, we obtain that the total number of Hamiltonian cycles in $G_n$ is at least~$\frac{(n-3)!}{12}$.
	
	Thus, the total number of Hamiltonian cycles up to automorphism is at least 
	\[\frac{(n-3)!}{12}\cdot \frac{1}{|\Aut(G_n)|}=\frac{(n-3)!}{12 \cdot 2n} \in 2^{\Theta(n \log n)}.\] 
\end{proof}


\section{Conclusion and future work}

In this paper, we initiated the study of the number of symmetry classes of Hamiltonian cycles in a graph. We mainly focused on graphs with a unique Hamiltonian cycle up to symmetry. We showed that this graph class contains \(d\)-regular graphs for every \(d \geq 2\) (\cref{rem:reg-graphs-in-H}) and broad families of 3-connected cubic graphs (\cref{thm:cub-graphs-in-H}).

Moreover, we studied Cartesian products and provided conditions on the factors for the product to be Hamiltonian-transitive (\cref{thm:cart-prod}). If one factor is $K_2$ and the other Hamiltonian, we gave a full characterization (\cref{thm:cartesian-product-K2}).

Further, by explicitly constructing Hamiltonian cycles, we proved that most Cayley graphs of abelian groups have more than one symmetry class of Hamiltonian cycles (\cref{thm:odd-cayley-graphs}, \cref{thm:cayley_even}). Based on these results, we conjecture the following.

\begin{conjecture}
	Let $\Gamma$ be a finite abelian group with generating set $S$ and let $G = \Cay(G,S)$. Then $G \in \mH$ if and only if \(G \in \{C_4 \bx K_2, C_n, K_n, K_{n,n}, C_k \bx K_2 \colon n,k \in \N, k \text{ odd}\}\).
\end{conjecture}

Finally, we showed that there are graphs with asymptotically as many symmetry classes of Hamiltonian cycles as possible (\cref{prop:many-ham-cyc}).

Overall, we believe that the enumeration and analysis of substructures in graphs up to symmetry is an important topic that should be further pursued. For example, we suggest the following research directions: 

\begin{itemize}
	\item investigate Hamiltonian transitivity for further classes of graphs, for instance, further Cayley graphs that are known to be Hamiltonian,
	\item study the computational aspects of this problem, for instance, by developing algorithms for enumerating Hamiltonian cycles up to symmetry, 
	\item explore analogous notions of transitivity for other substructures in graphs, such as perfect matchings.
\end{itemize}


\bibliography{TransitiveHamiltonian-ArXiv.bib}

\end{document}